\documentclass[11pt,a4paper,twoside]{article}
\usepackage[hmarginratio=1:1,bmargin=1.3in]{geometry}
\usepackage[frenchb]{babel}
\usepackage[utf8]{inputenc}
\usepackage[OT2,T1]{fontenc}
\usepackage{amsmath, amsthm}
\usepackage{amsfonts}
\usepackage{amscd}
\usepackage{amssymb}
\usepackage[all]{xy}
\usepackage{graphicx}
\usepackage{color}
\usepackage[nottoc, notlof, notlot]{tocbibind}
\usepackage{array}
\usepackage{url}
\usepackage{rotating}
\usepackage{fancyhdr}
\usepackage{fancyhdr}
\usepackage{titlesec}
\usepackage{xfrac} 
\usepackage{macro}
\title{L'espace ad\'elique d'un tore sur un corps de fonctions}
\author{David Harari et Diego Izquierdo}
\date{\today}

\titleformat{\section}[hang]{\center\Large\bf}{\thesection.}{0.5cm}{}

\DeclareSymbolFont{cyrletters}{OT2}{wncyr}{m}{n}
\DeclareMathSymbol{\Sha}{\mathalpha}{cyrletters}{"58}
\DeclareMathSymbol{\Brusse}{\mathalpha}{cyrletters}{"42}
\addto\captionsfrenchb{}

\theoremstyle{plain}
\newtheorem{theorem}{Théorème}[section]
\newtheorem{lemma}[theorem]{Lemme}
\newtheorem{proposition}[theorem]{Proposition}
\newtheorem{corollary}[theorem]{Corollaire}

\theoremstyle{definition}
\newtheorem{remarque}[theorem]{Remarque}

\newtheorem{notation}[theorem]{Notation}

\newtheorem{conclusion}[theorem]{Conclusion}

\newtheoremstyle{hypo}  % follow `plain` defaults but change HEADSPACE.
  {\topsep}   % ABOVESPACE
  {\topsep}   % BELOWSPACE
  {\itshape}  % BODYFONT
  {1.5ex}       % INDENT (empty value is the same as 0pt)
  {\bfseries} % HEADFONT
  {)}         % HEADPUNCT
  {8pt plus 1pt minus 1pt}  % HEADSPACE. `plain` default: {5pt plus 1pt minus 1pt}
  {}          % CUSTOM-HEAD-SPEC
\theoremstyle{hypo}

\newtheoremstyle{hypol}  % follow `plain` defaults but change HEADSPACE.
  {\topsep}   % ABOVESPACE
  {\topsep}   % BELOWSPACE
  {\itshape}  % BODYFONT
  {1.5ex}       % INDENT (empty value is the same as 0pt)
  {\bfseries} % HEADFONT
  {)$_{\ell}$}         % HEADPUNCT
  {8pt plus 1pt minus 1pt}  % HEADSPACE. `plain` default: {5pt plus 1pt minus 1pt}
  {}          % CUSTOM-HEAD-SPEC
\theoremstyle{hypol}

\newtheoremstyle{DP}  % follow `plain` defaults but change HEADSPACE.
  {\topsep}   % ABOVESPACE
  {\topsep}   % BELOWSPACE
  {\itshape}  % BODYFONT
  {}       % INDENT (empty value is the same as 0pt)
  {\bfseries} % HEADFONT
  {.}         % HEADPUNCT
  {8pt plus 1pt minus 1pt}  % HEADSPACE. `plain` default: {5pt plus 1pt minus 1pt}
  {}          % CUSTOM-HEAD-SPEC
\theoremstyle{DP}

\begin{document}

\maketitle

\small

\textbf{Résumé.} Soient $k$ un corps de caractéristique 0 et $K$ le corps des fonctions d'une $k$-courbe projective lisse géométriquement intègre $X$. Soit $T$ un $K$-tore. Dans cet article, on cherche à étudier l'espace des points adéliques $T(S,\A_K)$ de $T$ hors d'un ensemble fini $S$ de points fermés de $X$. On commence par montrer que le groupe $T(K)$ des points rationnels de $T$ est toujours fermé discret dans $T(S,\A_K)$. On décrit ensuite le quotient $T(\emptyset,\A_K)/T(K)$ dans chacun des trois cas suivants: $k$ corps algébriquement clos, $k=\mathbb{C}((t))$ et $k$ corps $p$-adique. \\

\textbf{Abstract.} Let $k$ be a field of characteristic 0 and let $K$ be the function field of a smooth projective geometrically integral $k$-curve $X$. Let $T$ be a $K$-torus. In this article, we aim at studying the space of adelic points $T(S,\A_K)$ of $T$ outside a finite set $S$ of closed points of $X$. We start by proving that the group $T(K)$ of rational points of $T$ is always discrete (hence closed) in $T(S,\A_K)$. We then describe the quotient $T(\emptyset,\A_K)/T(K)$ in each of the following three cases: 
$k$ is an algebraically closed field, $k$ is the field of Laurent series $\mathbb{C}((t))$, and $k$ is a $p$-adic field.

\normalsize

\section{Introduction}

Soit $k$ un corps de caract\'eristique $0$. On consid\`ere le corps 
des fonctions $K$ d'une courbe projective et lisse 
$X$ d\'efinie sur $k$, et les divers compl\'et\'es $K_v$ (d'anneau des entiers
$\calo_v$) du corps $K$ 
par rapport aux valuations induites par les points ferm\'es de $X$. 
Soit maintenant $G$ un $K$-groupe alg\'ebrique lin\'eaire connexe pour lequel 
on choisit un mod\`ele lisse ${\mathcal G}$ au-dessus d'un ouvert de 
Zariski non vide $U$ de $X$. 
L'{\it espace ad\'elique} de $G$, not\'e $G(\A_K)$,  
est le produit restreint $\prod_{v \in X^{(1)}} '  G(K_v)$ sur l'ensemble 
$X^{(1)}$ des points ferm\'es de $X$, par rapport aux 
${\mathcal G}(\calo_v)$ (il est ind\'ependant du mod\`ele choisi). 
On peut aussi enlever un nombre fini $S$ de points ferm\'es et consid\'erer 
$G(S, \A_K):=\prod_{v \in U^{(1)}} '  G(K_v)$ quand $U=X-S$ est un ouvert 
de Zariski non vide de $X$.

\smallskip

Ces derni\`eres ann\'ees, plusieurs travaux se sont pench\'es sur 
des questions arithm\'etiques classiques li\'ees aux groupes alg\'ebriques 
sur ces corps de fonctions, comme le principe local-global pour leurs 
espaces principaux homog\`enes ou encore l'approximation faible 
(c'est-\`a-dire la densit\'e de l'ensemble des points rationnels $G(K)$ 
dans $\prod_{v \in X^{(1)}} G(K_v)$, ce dernier groupe \'etant \'equip\'e de la 
topologie produit). Mentionnons par exemple 
\cite{gil}, qui \'etablit l'approximation faible pour $k$ alg\'ebriquement 
clos, les travaux \cite{dhsza1} et \cite{dhsza2} qui traitent du cas 
o\`u $k$ est $p$-adique, l'article \cite{dhct} qui concerne le cas 
$k=\CC((t))$, et les diverses 
g\'en\'eralisations \cite{diego1}, \cite{diego2} au cas o\`u $k$ est un 
corps local sup\'erieur.

\smallskip

L'origine de ce travail est l'article \cite{colejm} de J.-L. 
Colliot-Th\'el\`ene, lequel s'int\'eresse quand $k=\CC$ 
\`a la question de l'{\it approximation 
forte} pour $G$, c'est-\`a-dire \`a la densit\'e de $G(K)$ dans 
un espace ad\'elique $G(S,\A_K)$, ce dernier \'etant \'equip\'e de la topologie 
de produit restreint (et non plus comme pour l'approximation 
faible de celle induite par la topologie produit sur 
$\prod_{v \in U^{(1)}} G(K_v)$). 
Les deux r\'esultats principaux de \cite{colejm} 
sont~: d'une part, la validit\'e de 
l'approximation forte quand $S \neq \emptyset$ et $G$ est semi-simple; d'autre 
part, des exemples o\`u cette approximation forte ne vaut 
pas quand $G$ est un tore. Notre but ici est de pr\'eciser cet 
\'enonc\'e sur les tores dans plusieurs directions. 

\smallskip

Notre premier r\'esultat (th\'eor\`eme~\ref{gentheo}) est le suivant~:
nous montrons que pour tout corps $k$ de caract\'eristique z\'ero, 
tout ensemble fini $S$ de points ferm\'es de $X$, et tout $K$-tore $T$, 
l'image de $T(K)$ dans $T(S,\A_K)$ est un sous-groupe discret (donc 
ferm\'e) de $T(S,\A_K)$. En particulier l'approximation forte ne vaut 
jamais pour un tore de dimension $>0$.

\smallskip

Ainsi, dans notre situation, le d\'efaut d'approximation forte 
est simplement le quotient $A(S,T)$ de 
l'espace ad\'elique $T(S,\A_K)$ par $T(K)$; quand $T=\G$ et $S=\emptyset$, 
ce quotient $C_K=(\prod_{v \in X^{(1)}} ' K_v^* )/K^*$ 
est juste le {\it groupe des classes d'id\`eles} de $K$,  
notion qui est classique pour un corps de nombres ou le corps des fonctions 
d'une courbe sur un corps fini. Par analogie avec la th\'eorie du corps de
classes global, on aimerait maintenant mieux comprendre la structure 
de $T(S,\A_K)$ et notamment relier $A(\emptyset,T)$ \`a une obstruction 
de r\'eciprocit\'e \`a l'approximation forte. Plus pr\'ecis\'ement, 
on cherche \`a construire une fl\`eche fonctorielle $r: A(\emptyset,T) \to 
P$ d'image dense, o\`u $P$ est un groupe ab\'elien 
profini (dual d'un groupe discret de 
torsion d\'efini via la cohomologie galoisienne du tore $T$). Comme 
dans le cas du groupe des classes d'id\`eles d'un corps de nombres 
(o\`u $P$ est juste le groupe de Galois ab\'elien de $K$), on ne peut 
esp\'erer que $r$ soit injective mais seulement que son noyau 
soit un groupe ab\'elien divisible (le groupe profini $P$ n'ayant pas 
d'\'el\'ement infiniment divisible non nul).

\smallskip

Pour tout groupe topologique ab\'elien $B$, notons $B^D$ son dual, 
c'est-\`a-dire le groupe $\Hom_c(B,\Q/\Z)$ 
des homomorphismes continus de $B$ dans $\Q/\Z$ 
(quand la topologie sur $B$ n'est pas pr\'ecis\'ee, on convient que $B$ est 
muni de la topologie discr\`ete). 
Dans les trois situations consid\'er\'ees dans les articles 
ant\'erieurs, nous obtenons alors les r\'esultats suivants~:

\smallskip

a) Pour $k$ alg\'ebriquement clos (auquel cas $K$ est de dimension 
cohomologique ${\rm cd}(K)=1$), le groupe $A(S,T)$ est divisible
si $S \neq \emptyset$ (on a donc en quelque sorte ``approximation 
forte modulo divisible''). Ce r\'esultat vaut encore si on remplace le tore 
$T$  par un groupe alg\'ebrique lin\'eaire connexe quelconque 
(th\'eor\`eme~\ref{lineairetheo}).
%% passage modifie
De plus, le quotient $\overline{{A(\emptyset,T)}}$ 
de $A(\emptyset,T)$ par son sous-groupe divisible maximal est 
un groupe de type fini dont on peut calculer 
le rang (th\'eor\`eme~\ref{rangtheo}). 
Nous \'etablissons aussi une suite de 
Poitou-Tate pour $T$ (c'est le cas "$d=-1$" du 
th\'eor\`eme~3.20 de \cite{diego1}, qui avait \'et\'e fait pour les modules
finis mais pas pour les tores), et en d\'eduisons 
une fl\`eche $r : {A(\emptyset,T)} \to 
H^0(K,\hat T \otimes \Q/\Z(-1) )^D$ d'image dense et de noyau divisible
 (th\'eor\`eme~\ref{kalgclos}), o\`u $\hat T$ est le module des caract\`eres
de $T$.

\smallskip

b) Pour $k=\CC((t))$ (auquel cas ${\rm cd}(K)=2$, comme pour un corps de
fonctions sur un corps fini), 
on a encore (th\'eor\`eme~\ref{cttheo}) une application de r\'eciprocit\'e 
$r : {A(\emptyset,T)} \to H^2(K,\hat T)^D$ de noyau divisible, 
mais contrairement au cas o\`u $k$ est fini (auquel cas lorsque  
$T=\G$, l'image $I={\rm Im} \, r$ est m\^eme dense), l'adh\'erence 
$I_{\rm adh}$ de $I$ n'est pas d'indice fini dans $H^2(K,\hat T)^D$.
Plus pr\'ecis\'ement, le quotient $H^2(K,\hat T)^D/I_{\rm adh}$ 
est isomorphe au dual d'un groupe de 
Tate-Shafarevich $\Sha^2(K,\hat T)$, qui est infini en g\'en\'eral.
Ce groupe est d\'efini par la formule 
habituelle $\Sha^2(K,\hat T):=
\ker [H^2(K,\hat T) \to \prod_{v \in X^{(1)}} H^2(K_v,\hat T)]$.

\smallskip

c) Pour $k$ corps $p$-adique (auquel cas ${\rm cd}(K)=3$), on a un 
r\'esultat analogue (corollaire~\ref{qpcor}) 
en rempla\c cant $H^2(K,\hat T)$ par 
$H^2(K,T')$, o\`u $T'$ est le tore dual de $T$; mais ici, l'adh\'erence 
de ${\rm Im} \, r$ est \`a nouveau d'indice fini dans 
$H^2(K,T')^D$, le quotient \'etant le groupe fini $\Sha^1(K,T):=
\ker [H^1(K,T) \to \prod_{v \in X^{(1)}} H^1(K_v, T)]$. 

\smallskip

On comparera les r\'esultats b) et c) 
avec les calculs du d\'efaut d'approximation 
faible pour un tore de \cite{dhct} et \cite{dhsza2}. Pour l'approximation 
forte, ce sont les groupes $H^2(K,\hat T)$ et $H^2(K,T')$ qui apparaissent 
respectivement, et non plus leur sous-groupe des \'el\'ements 
presque partout localement triviaux. En ce sens, la situation est assez 
similaire \`a celle des corps de nombres (voir \cite{dhant}). 

\smallskip 

Nous compl\'etons notre \'etude par une description plus pr\'ecise de
la partie de torsion du sous-groupe divisible maximal de $A(\emptyset,T)$, 
que nous relions \`a la cohomologie galoisienne de $\hat T$ (corollaires 
\ref{corcc}, \ref{corcct} et \ref{corccqp}). 
La structure de $A(\emptyset,T)$ est ainsi compl\`etement 
\'elucid\'ee dans les trois cas consid\'er\'es.

\smallskip

Par ailleurs, une difficult\'e 
(notamment dans le cas c)) est qu'il n'est pas a priori \'evident que 
le sous-groupe divisible maximal et le sous-groupe des \'el\'ements 
infiniments divisibles co\"{\i}ncident pour $A(\emptyset,T)$; nous 
commen\c cons donc par quelques g\'en\'eralit\'es sur ces deux notions.

\section{Quelques lemmes sur les groupes ab\'eliens} 

Soit $B$ un groupe ab\'elien. On note $B_{\div}$ le plus grand 
sous-groupe divisible de $B$, qui est la somme de tous les sous-groupes 
divisibles de $B$. On rappelle que comme tout sous-groupe divisible, 
c'est un facteur direct de $B$. On pose $\ov B=B/B_{\div}$. 
Par ailleurs, on note 
$B_{\infty-\div}=\bigcap_{n \in \NN^*} nB$ le sous-groupe des \'el\'ements 
infiniment divisibles de $B$. On a clairement $B_{\div} \subset 
B_{\infty-\div}$, mais l'inclusion peut \^etre stricte si on ne fait 
pas d'hypoth\`ese suppl\'ementaire sur $B$, autrement dit $B_{\infty-\div}$ 
peut ne pas \^etre un groupe divisible (cf. \cite{fuc}, exemple de la page 150). 
Si $p$ est un nombre 
premier et $B_{\div_p}$ d\'esigne 
le plus grand sous-groupe $p$-divisible de $B$ (c'est le plus grand 
sous-groupe sur lequel la multiplication par $p$ est surjective), on voit
imm\'ediatement que $B_{\div}$ est l'intersection pour $p$ premier 
des $B_{\div_p}$. On a $x \in B_{\div_p}$ si et seulement s'il existe 
une suite infinie (et pas seulement une suite finie de longueur 
arbitrairement grande) $(x_n)$ d'\'el\'ements de $B$ 
avec $x_1=x$ et $x_n=p x_{n+1}$ pour tout 
$n \in \NN^*$. On en d\'eduit ais\'ement 
que $x \in B_{\div}$ si et seulement s'il existe une suite infinie 
$(x_n)$ d'\'el\'ements de $B$ 
avec $x_1=x$ et $x_n=mx_{mn}$ pour tous $m,n \in \NN^*$. 

\smallskip

Si $B$ est un groupe de 
torsion de type cofini (i.e. la $n$-torsion $_n B$ est finie pour tout
$n \in \NN^*$), il est classique que $B_{\div}=B_{\infty-\div}$ car 
on sait alors que $B$ est somme directe sur tous les nombres premiers 
$\ell$ de groupes de la forme $F \oplus (\Q_{\ell}/\Z_{\ell})^r$, o\`u 
$F$ est un $\ell$-groupe ab\'elien fini.

\smallskip

Dans cette section, on regroupe quelques r\'esultats g\'en\'eraux ayant 
trait notamment au comportement de $B_{\div}$ et $B_{\infty-\div}$ dans 
les suites exactes. On note $B_{\rm tors}=\bigcup_{n >0} \, _n B$ le sous-groupe
de torsion de $B$.

\begin{lemma}\label{easy}
Considérons une suite exacte de groupes abéliens:
\begin{equation}
0\rightarrow A \xrightarrow[]{f} B \xrightarrow[]{g} C \rightarrow 0, \label{se1}
\end{equation}
telle que $C_{\text{div}}=0$. Alors la suite (\ref{se1}) induit une suite exacte:
\begin{equation}
0\rightarrow \overline{A} \xrightarrow[]{\overline{f}} \overline{B} \xrightarrow[]{\overline{g}} C \rightarrow 0.\label{se2}
\end{equation}
\end{lemma}

\begin{proof}
Le seul point non trivial à vérifier est l'injectivité de 
$\overline{f}$. Donnons-nous $x \in \text{Ker} (\overline{f})$. 
Soit $\tilde{x} \in A$ un relèvement de $x$. On sait alors que 
$\tilde{y}:=f(\tilde{x}) \in B_{\text{div}}$. 
Comme $B_{\div}$ est divisible, il existe 
une famille $(\tilde{y}_n)_{n\geq 1}$ à valeurs dans 
$B_{\text{div}}$ telle que $\tilde{y}=\tilde{y}_1$ et 
$\tilde{y}_n=m\tilde{y}_{mn}$ pour tous $m,n\geq 1$. 
De plus, comme $C_{\text{div}}=0$, on a $g(\tilde{y}_n)=0$, et donc, d'après 
la suite exacte (\ref{se1}), il existe $\tilde{x}_n\in A$ vérifiant 
$f(\tilde{x}_n)=\tilde{y}_n$. Par injectivité de $f$ on déduit que 
$\tilde{x}=\tilde{x}_1$ et $\tilde{x}_n=m\tilde{x}_{mn}$ pour tous $m,n\geq 1$. Autrement dit, $\tilde{x} \in A_{\text{div}}$ et $x=0$.

\end{proof}

\begin{lemma}\label{infdiv0}
Considérons une exacte de groupes abéliens:
$$0\rightarrow A \xrightarrow[]{f} B \xrightarrow[]{g} C\rightarrow 0$$
où $C$ est un groupe d'exposant fini $e$. Alors 
$B_{\infty-\div}=A_{\infty-\div}$.
\end{lemma}

\begin{proof}
Soit $x \in B_{\infty-\div}$. 
Pour $n\geq 1$, écrivons $x=nx_n$ avec $x_n \in B$.
Pour $n\geq 1$, on a $x=n\cdot (ex_{ne})$. Le groupe $C$ étant d'exposant $e$, il existe $y\in A$ ainsi que $y_n \in A$ tels que $f(y)=x$ et $f(y_n)=ex_{ne}$. On a alors $y=ny_n$. Cela prouve que $x \in A_{\infty-\div}$, comme on voulait.

\end{proof}

\begin{lemma}\label{infdiv}
Considérons une exacte de groupes abéliens:
$$0\rightarrow A \xrightarrow[]{f} B \xrightarrow[]{g} C \rightarrow 0.$$
Supposons que le groupe $C$ s'insère dans une suite exacte:
$$0\rightarrow C' \xrightarrow[]{h} C \xrightarrow[]{i} E\rightarrow 0$$
où $C'$ est un groupe sans torsion avec $C'_{\infty-\div}=0$ 
et $E$ est un groupe d'exposant fini $e$. Alors 
$B_{\infty-\div}=A_{\infty-\div}$.

\end{lemma}

\begin{proof}
Soit $x \in B_{\infty-\div}$. 
Pour $n\geq 1$, écrivons $x=nx_n$ avec $x_n \in B$.
\begin{itemize}
\item[$\bullet$] Montrons d'abord que $x$ est dans l'image de $f$. 
L'élément $g(x)$ de $C$ est infiniment divisible. 
Donc d'après le lemme \ref{infdiv0}, il existe $y\in C'_{\infty-\div}$ 
tel que $h(y)=g(x)$. Par hypothèse, $y=0$, et donc $g(x)=0$: autrement dit, il existe $z\in A$ tel que $f(z)=x$.
\item[$\bullet$] Soit $n\geq 1$. Remarquons que $ng(ex_{ne})=g(x)=0$ et 
d'autre part $g(ex_{ne})\in \text{Im}(h)$ car $E$ est d'exposant $e$. Comme $C'$ est sans torsion, on déduit que $g(ex_{ne})=0$. Il existe alors $z_n \in A$ tel que $f(z_n)=ex_{ne}$. On a $z=nz_n$ pour chaque $n$, donc $z$ est infiniment divisible dans $A$.
\end{itemize}
\end{proof}

\begin{lemma}\label{ud}
Considérons une suite exacte de groupes abéliens:
$$0 \rightarrow A \xrightarrow[]{f} B \xrightarrow[]{g} C \xrightarrow[]{h} C' \rightarrow 0.$$
Supposons que $A$ et $C'$ sont d'exposant fini et que le sous-groupe divisible maximal $B_{\div}$ 
de $B$ est uniquement divisible. Alors $C_{\div}$ est uniquement divisible.
\end{lemma}

\begin{proof}
Comme $C'$ est d'exposant fini, $C_{\div}$ est contenu dans l'image de $g$. 
De plus $B_{\div} \subset g^{-1}(C_{\div})$, ce qui montre que 
$B$ et $g^{-1}(C_{\div})$ ont tous deux pour sous-groupe divisible 
maximal $B_{\div}$.
Quitte à remplacer $C$ par $C_{\text{div}}$ et $B$ par 
$g^{-1}(C_{\text{div}})$, on peut donc supposer que $C'=0$ et 
que $C$ est divisible. Donnons-nous $x\in C_{\text{tors}}$ et 
considérons $(x_n)$ une suite d'éléments de $C$ telle que 
$x_1=x$ et $x_n=mx_{mn}$ pour tous $m,n>0$. Pour chaque $n\geq 1$, 
soit $y_n\in B$ tel que $g(y_n)=x_n$. Quels que soient les entiers 
$m$ et $n$, il existe $z_{m,n}\in A$ tel que $y_n=my_{mn} + z_{m,n}$. 
Ainsi, si $e$ désigne l'exposant de $A$, on a $ey_n=mey_{mn}$. 
En particulier, $ey_{e} \in B_{\text{div}}$ car la suite infinie $(u_n)$ 
d\'efinie par $u_n=ey_{en}$ v\'erifie $u_1=e y_e$ et $mu_{mn}=mey_{m(ne)}=
ey_{en}=u_n$ pour tous $m,n \in \NN^*$.
Comme $ey_{e}$ est de torsion (car $g(ey_e)$ l'est, ainsi que $A$), 
on en déduit que $ey_e=0$, puis que $y_1 \in A$ (via l'\'egalit\'e 
$y_1=ey_e+z_{e,1}$) et enfin $x=0$. Cela achève la preuve.
\end{proof}

\begin{lemma}\label{fini}
Considérons une suite exacte de groupes abéliens:
$$0 \rightarrow A \xrightarrow[]{f} B \xrightarrow[]{g} C \rightarrow 0.$$
Supposons que $A$ est fini et que $B_{\infty-\div}=B_{\div}$. 
Alors $C_{\infty-\div}=C_{\div}$.
\end{lemma}

\begin{proof}
Soit $x\in C_{\infty-\div}$. Pour $n >0$, on se donne $x_n  \in C$ tel que $x=nx_n$. On considère ensuite $y \in B$ (resp. $y_n \in B$) tel que $g(y)=x$ (resp. $g(y_n)=x_n$). Pour chaque $n >0$, il existe $t_n\in A$ tel que $y=ny_n+t_n$. Soit $t\in A$ tel que $t_{n!}=t$ pour une infinité de valeurs de $n$. Alors $y-t$ est infiniment divisible dans $B$. Par hypothèse, cela implique que $y-t\in B_{\text{div}}$, et donc que $x=g(y-t)\in C_{\text{div}}$.
\end{proof}

\begin{lemma}\label{tcof}
Soit $B$ un groupe abélien tel que $B_{\text{tors}}$ est de torsion de type cofini. Alors $B_{\infty \text{-div}} = B_{\text{div}}$.
\end{lemma}

\begin{proof}
 Écrivons $B_{\text{tors}} = (\bigoplus_{p} F_p) \oplus D$ où $p$ décrit l'ensemble des nombres premiers, $F_p$ est un $p$-groupe abélien fini pour chaque $p$ et $D$ est divisible. On a des suites exactes:
\begin{gather} \label{e1}
 0 \rightarrow D \rightarrow B \rightarrow B/D \rightarrow 0,\\
 \label{e2}
 0 \rightarrow \bigoplus_p F_p \rightarrow B/D \rightarrow B/B_{\text{tors}} \rightarrow 0.
 \end{gather}
La suite (\ref{e1}) est scindée car $D$ est divisible. Il suffit donc de démontrer que $(B/D)_{\infty \text{-div}} = (B/D)_{\text{div}}$. Soient $x\in (B/D)_{\infty \text{-div}} $ et $p$ un nombre premier. Pour chaque $n\geq 1$, considérons $x_n  \in B/D$ tel que $x=p^nx_n$. Notons $y$ (resp. $y_n$) l'image de $x$ (resp. $x_n$) dans $B/B_{\text{tors}}$. On a alors $y=p^ny_n$ pour $n\geq 1$, et donc $p^{n}(py_{n+1}-y_n)=0$. Comme $B/B_{\text{tors}}$ n'a pas de torsion, on déduit que $py_{n+1}=y_n$. Par conséquent, pour $n\geq 1$, on a $px_{n+1}=x_n+t_n$ pour un certain $t_n \in \bigoplus_p F_p$. En remarquant que $p^nt_n=0$, on obtient que $t_n\in F_p$. Soit $s\geq 1$ tel que $p^sF_p=0$. On a alors $p^{s+1}x_{n+1}=p^sx_n$ pour chaque $n\geq 1$. En posant $z_n=p^sx_{s+n-1}$, on obtient que $z_1=x$ et $pz_{n+1}=z_n$ pour tout $n\geq 1$. Cela pouvant être fait pour chaque premier $p$, on déduit que $x\in (B/D)_{\text{div}}$ (car $x$ est dans 
le sous-groupe $p$-divisible maximal de $B/D$ pour tout $p$ premier), 
ce qui achève la preuve.
\end{proof}

Pour tout groupe ab\'elien $A$,
on pose $A_{\wedge}:=\varprojlim_{n >0} (A/nA)$. Si tous les
$A/nA$ sont finis, alors $A_{\wedge}$ n'est autre que le compl\'et\'e
profini $A^{\wedge}$ de $A$. De plus, on a $A_{\wedge}=A$ si $A$ est 
profini.

\begin{lemma} \label{completion}
Soit $C$ un groupe ab\'elien profini. Soient $A$ et $B$ des groupes 
ab\'eliens. 

\smallskip

a) Soit $f : B \to C$ un morphisme tel que le morphisme induit 
$B_{\wedge} \to C$ soit surjectif. Alors $f(B)$ est dense dans 
$C$.

\smallskip

b) Soit
$$A \stackrel{i}{\to} B \to C$$
un complexe de groupes ab\'eliens avec $C$ profini. On suppose que 
le complexe associ\'e 
$$A_{\wedge} \to B_{\wedge} \to C$$
est une suite exacte. Soit $E$ le quotient de $B$ par $i(A)$. 
Alors le noyau de l'application induite $u : E \to C$ est 
$E_{\infty-\div}$.
\end{lemma}

\begin{proof}

a) L'hypoth\`ese implique que pour tout $n >0$, l'application 
$B/nB \to C/nC$ induite par $f$ est surjective. Comme $C$ est profini,
tout sous-groupe ouvert $U$ de $C$ est d'indice fini, donc 
contient $nC$ pour un certain $n >0$. Ceci montre que si $x \in C$, 
alors l'ouvert $x+U$ rencontre $f(B)$ puisqu'on peut \'ecrire
$x=f(y)+x'$ avec $x' \in nC$. Ceci montre que $f(B)$ est dense 
dans $C$.

\smallskip

b) Comme $C$ est profini, il est limite projective des $C/nC$ et 
on a donc $C_{\infty-\div}=0$, ce qui montre que $E_{\infty-\div} \subset 
\ker u$. En sens inverse, soit $x \in \ker u$, qu'on rel\`eve en 
$y \in B$. Soit $n >0$. Alors l'image de $y$ dans $B_{\wedge}$ provient 
de $A_{\wedge}$, ce qui implique que l'image de $y$ dans $B/nB$ provient 
de $A/nA$. Il existe donc $a \in A$ tel que $i(a)=y+nb$ avec $b \in B$. 
Ainsi l'image $x$ de $y$ dans $E=B/i(A)$ est divisible par $n$. 
Finalement $x \in E_{\infty-\div}$.
\end{proof}

\section{Cas général}

Dans cette section, on consid\`ere le corps des fonctions $K$ 
d'une courbe projective lisse 
géométriquement intègre $X$ sur un corps de base $k$ de caractéristique 0. 
Nous allons montrer que sans aucune hypoth\`ese sur $k$, le groupe 
des points rationnels $T(K)$ d'un $K$-tore est discret dans l'espace 
ad\'elique $T(\A_K,S)=\prod_{v \not \in S} ' T(K_v)$, et ce pour 
tout ensemble fini $S$ de points ferm\'es de $X$. On commence par 
le cas du tore d\'eploy\'e $T=\G$.

\smallskip

Soient $U$ un ouvert non vide de $X$ et $S$ l'ensemble fini 
$X-U$. On note $k[U]$ l'anneau des fonctions r\'eguli\`eres sur $U$ 
et $k[U]^*=H^0(U,\G)$ le groupe des fonctions inversibles sur $U$.
Les valuations associ\'ees aux points 
ferm\'es de $U$ induisent une application $${\rm val} : 
\G(A,S)=\prod_{v\in U^{(1)}}' K_v^* \to \bigoplus_{v\in U^{(1)}} 
\mathbb{Z}.$$ On \'equipe $\bigoplus_{v\in U^{(1)}} \mathbb{Z}$ de 
la topologie discr\`ete et $\G(A,S)$ de la topologie de produit restreint 
associ\'ee aux topologies $v$-adiques.
Considérons le diagramme:\\

\centerline{\xymatrix{
K^* \ar[r]^-{\iota}\ar[rd] & \prod_{v\in U^{(1)}}' K_v^* \ar[d]^{{\rm val}}\\
 & \bigoplus_{v\in U^{(1)}} \mathbb{Z}.
}}

\begin{lemma} \label{val}
 Le morphisme $ \iota(K^*) \rightarrow \bigoplus_{v\in U^{(1)}} \mathbb{Z}$ est continu quand on \'equipe $\iota(K^*)$ de la topologie 
induite par celle de $\G(A,S)$.
 \end{lemma}

\begin{proof}
Il suffit de vérifier que le morphisme 
$\prod_{v\in U^{(1)}}' K_v^* \rightarrow \bigoplus_{v\in U^{(1)}} \mathbb{Z}$
induit par les valuations est continu. Cela découle immédiatement de la définition de la topologie produit restreint car le noyau 
$\prod_{v\in U^{(1)}} \mathcal{O}_v^*$ de ${\rm val}$ est un sous-groupe 
ouvert de $\prod_{v\in U^{(1)}}' K_v^*$.
\end{proof}

\begin{lemma}\label{unit}
Soit $v\in U^{(1)}$. Soit $\iota_v$ l'injection naturelle de 
$k[U]^*$ dans $K_v^*$. Alors 
$\iota_v (k[U]^*)$ est fermé discret dans $\calo_v ^*$ 
(et donc aussi dans $K_v^*$). 
\end{lemma}

\begin{proof}
Le groupe $k[U]^*/k^*$ est de type fini car on a une suite
exacte 
$$1 \to k^* \to k[U]^* \stackrel{\oplus {\rm val}_w}{\to} 
\bigoplus_{w \in S} \Z.$$

Du coup, si on pose $\mathcal{O}_v^1 := \{x \in \mathcal{O}_v | v(x-1)>0\}$ 
et si on note $k(v)$ le corps r\'esiduel de $\calo_v$, 
alors dans le groupe $K_v^* \cong \mathbb{Z} \times k(v)^* \times \mathcal{O}_v^1$, on a~:
$$\iota_v (k[U]^*) \subset k(v)^* \times \mathcal{O}_v^1, $$
et il suffit de d\'emontrer que $\iota_v (k[U]^*) \cap \mathcal{O}_v^1$
est discret vu que $k(v)^*$ est discret. Comme $\iota_v (k[U]^*) \cap \mathcal{O}_v^1$ s'injecte dans le groupe de type fini $\iota_v (k[U]^*/k^*)$, il
est lui m\^eme de type fini.
Étant donné que $\mathcal{O}_v^1 \cong k[[t]]$, il suffit de démontrer le lemme qui suit~: 

\begin{lemma}
Soit $H$ un sous-groupe de type fini du groupe additif 
$k[[t]]$. Alors $H$ est fermé discret dans $k[[t]]$.
\end{lemma}

Notons que c'est pr\'ecis\'ement ce r\'esultat qui est clairement faux 
si on remplace $k[[t]]$ par l'anneau des entiers d'un corps complet 
pour une valuation discr\`ete
d'in\'egale caract\'eristique comme $\Q_p$~: par exemple, 
le sous-groupe de type fini $\Z$ n'est pas ferm\'e dans $\Z_p$.

\smallskip

{\it D\'emonstration du lemme.}
Soit $(z_1,...,z_m)$ une famille génératrice de $H$ (on peut m\^eme 
supposer que c'est une base puisque $H$ est libre, \'etant de type
fini et sans torsion).
Soit $(x_n)_n$ une suite à valeurs dans $H$ qui converge vers un certain $y \in k[[t]]$. Écrivons: $$x_n = a_1^{(n)}z_1+...+a_m^{(n)}z_m$$ avec $a_1^{(n)},..., a_m^{(n)}$ entiers. Écrivons aussi: 
$$z_j = \sum_{i\geq 0} z_i^{(j)} t^i,\;\;\;
y=\sum_{i\geq 0} y_it^i.$$

On a donc 
$$x_n=\sum_{i \geq 0} (\sum_{r=1} ^m a_r ^{(n)} z_i^{(r)} )t^i.$$
Par d\'efinition de la topologie sur $k[[t]]$, la convergence 
de la suite $(x_n)$ vers $y$ signifie que 
pour chaque $i\geq 0$, il existe $n_i\geq 1$ tel que, 
pour tout $n\geq n_i$, on a:
$$y_i = a_1^{(n)}z_i^{(1)} + ... + a_m^{(n)}z_i^{(m)}.$$
Notons:
$$M_s := \{(b_1,...,b_m) \in \mathbb{Q}^m | \forall i\in \{0,1,...,s\}, y_i = b_1z_i^{(1)} + ... + b_mz_i^{(m)}\},$$
de sorte que $(a_1^{(n)},..., a_m^{(n)}) \in M_{s}$ pour $n\geq \max \{ n_0,...,n_s\}$. On remarque que les $M_s$ forment une suite décroissante de sous-espaces affines non vides 
de $\mathbb{Q}^m$. Par conséquent, si $M_{\infty} := \{(b_1,...,b_m) \in \mathbb{Q}^m |  y = b_1z_1 + ... + b_mz_m\}$, il existe $s_0 \geq 0$ tel que:
$$M_{s_0} = M_{s_0+1} = ... = M_{\infty}.$$
En notant $N = \max \{ n_0,...,n_{s_0}\}$, on a donc $(a_1^{(n)},..., a_m^{(n)}) \in M_{\infty}$ pour $n\geq N$. La suite $(x_n)$ est donc stationnaire, ce qui achève la preuve.
\end{proof}

\begin{proposition}\label{g_m}
L'image $\iota(K^*)$ de $K^*$  
dans $\prod_{v\in U^{(1)}}' K_v^*$ en est un sous-groupe discret (donc 
ferm\'e).
    \end{proposition}
    
    \begin{proof}

D'apr\`es le lemme~\ref{val}, le sous-groupe $\iota(k[U]^*)$ est un voisinage 
ouvert de $\{1 \}$ dans $\iota(K^*)$. Choisissons $v \in X^{(1)}$. 
D'apr\`es le lemme~ \ref{unit}, $\iota_v(k[U]^*)$ est discret dans 
$K_v^*$, donc a fortiori $\iota(k[U]^*)$ est un sous-groupe 
discret de $\prod_{v\in U^{(1)}}' K_v^*$ (la topologie induite 
par $\prod_{v\in U^{(1)}}' K_v^*$ est au moins 
aussi fine que celle induite par $K_v^* $).
Comme $\iota(k[U]^*)$ est un voisinage ouvert de $\{1 \}$ dans 
$\iota(K^*)$, cela montre que $\{ 1\}$ est ouvert dans $\iota(K^*)$,
autrement dit $\iota(K^*)$ est un groupe discret.

    \end{proof}

On va maintenant traiter le cas d'un tore quelconque. Rappelons 
qu'un $K$-tore $T$ est {\it quasi-trivial} si $\hat T$ est un 
module galoisien de permutation (i.e. il existe une base 
de $\hat T$ en tant que groupe ab\'elien qui est stable par l'action 
de Galois). De fa\c con \'equivalente, cela signifie que 
$T$ est isomorphe \`a un produit de tores de la forme 
$R_{E_i/K} \G$ (o\`u $R_{E_i/K}$ d\'esigne la restriction de Weil de  
$E_i$ \`a $K$, $E_i$ \'etant une extension finie de corps de $K$). 
Tout tore se plonge dans 
un tore quasi-trivial, car tout module galoisien de type fini est 
quotient d'un module de permutation.
    
    \begin{corollary}\label{qt}
    Soit $R$ un $K$-tore quasi-trivial. L'image de l'injection diagonale:
    $$R(K)\rightarrow \prod_{v\in U^{(1)}} ' R(K_v)$$
     est fermée discrète dans $\prod_{v\in U^{(1)}} ' R(K_v)$.
    \end{corollary}
    
    \begin{proof}
    Écrivons $R=R_{E/K}(\mathbb{G}_m)$ avec $E = E_1\times ... \times E_r$ et $E_1,...,E_r$ des extensions finies de $K$. Soient $X_1$,...,$X_r$ des courbes projectives lisses sur $K$ telles que $K(X_i)=E_i$ pour chaque $i$. Les extensions $E_i/K$ induisent bien sûr des morphismes $\pi_i: X_i \rightarrow X$. Notons $U_i = \pi_i^{-1}(U)$ pour chaque $i$. L'injection naturelle:
    $$R(K)\rightarrow \prod_{v\in U^{(1)}} ' R(K_v)$$
    s'identifie alors, via le lemme de Shapiro,  à l'injection naturelle:
    $$\prod_i E_i^*\rightarrow \prod_i \prod_{w\in U_i^{(1)}} ' E_{i,w}^*.$$
    D'après la proposition \ref{g_m}, son image est bien fermée discrète.
    \end{proof}
    
    \begin{theorem} \label{gentheo}
    Soit $T$ un $K$-tore. Considérons l'injection diagonale:
    $$\iota_T: T(K)\rightarrow \prod_{v\in U^{(1)}} ' T(K_v).$$
    Son image est fermée discrète dans $\prod_{v\in U^{(1)}} ' T(K_v)$.
    \end{theorem}
    
    \begin{proof}
    Soit $f: T \rightarrow R$ un $K$-morphisme injectif vers un 
tore quasi-trivial $R$. On a un diagramme commutatif à lignes exactes:\\
\centerline{\xymatrix{
0 \ar[r] & T(K)\ar[d]^{\iota_T} \ar[r] & R(K)\ar[d]^{\iota_R}\\
0 \ar[r] & \prod_{v\in U^{(1)}} ' T(K_v) \ar[r] & \prod_{v\in U^{(1)}} ' R(K_v)
}}

\smallskip

L'image de $\iota_R$ étant fermée discrète dans $\prod_{v\in U^{(1)}}' R(K_v)$ d'après le lemme \ref{qt}, il en est de même de l'image de $\iota_T$ dans $\prod_{v\in U^{(1)}}' T(K_v)$.
    \end{proof}
    
    \section{Cas {$k$} algébriquement clos}

Soit $K$ le corps des fonctions
d'une courbe projective et lisse $X$ sur un corps $k$
de caract\'eristique z\'ero. Fixons quelques notations, qui 
seront en vigueur jusqu'\`a la fin de l'article~:

\begin{notation}
Lorsque $T$ est un $K$-tore et $S$ une partie finie de $X^{(1)}$, on note:
$$A(S,T):= \left( \prod_{v\in U^{(1)}} ' T(K_v) \right) /T(K)$$
où $U=X \setminus S$ et on a identifi\'e $T(K)$ avec son image 
diagonale dans $\prod_{v\in U^{(1)}} ' T(K_v)$.
\end{notation}

\begin{notation}
Soit $T$ un $K$-tore.
Soit $\mathcal{T}$ un tore sur un ouvert non vide $U$ de $X$ étendant $T$. 

 On pose~: 
\begin{gather*}
   \mathbb{P}^0(T) := \prod_{v\in X^{(1)}}' T(K_v),\\
   \end{gather*}
(le produit restreint \'etant relatif aux ${\mathcal T}(\calo_v)$)
 et on munit $\mathbb{P}^0(T)$ de sa topologie de produit restreint.

Pour tout module galoisien $M$ sur $K$ et tout $i>0$, on pose~:
$$\Sha^i(K,M):=\ker [H^i(K,M) \to \prod_{v \in X^{(1)}} H^i(K_v,M)].$$

\end{notation}

\smallskip

Dans cette section, le corps $k$ est suppos\'e
alg\'ebriquement clos. Nous allons voir que dans ce cas, 
on peut d\'eterminer beaucoup plus pr\'ecis\'ement $A(S,T)$
(qui est aussi le d\'efaut d'approximation forte 
d'apr\`es le th\'eor\`eme~\ref{gentheo}).

\subsection{Structure de $A(S,T)$}
    
\begin{lemma}\label{qt2}
    Soit $S$ une partie finie de $X^{(1)}$. Notons $U = X\setminus S$ et considérons le groupe:
    $$A(S,\G):= \left( \prod_{v\in U^{(1)}} ' 
K_v^* \right) /K^*.$$
    \begin{itemize}
 \item[(i)]   Si $S = \emptyset$, le groupe $\overline{A(S,\G)}$ est isomorphe à $\mathbb{Z}$.
  \item[(ii)]  Si $S\neq \emptyset$, le groupe $A(S,\G)$ est divisible.
    \end{itemize}
\end{lemma}      
    
\begin{proof} En consid\'erant la valuation en chaque $v \in U^{(1)}$, on 
d\'efinit une fl\`eche surjective $\prod_{v\in U^{(1)}} ' K_v^* 
\to {\rm Div} \, U$, induisant une fl\`eche surjective 
$\prod_{v\in U^{(1)}} ' K_v^* \to \pic U$ dont le noyau est le 
sous-groupe $K^* \prod_{v\in U^{(1)}} \mathcal{O}_v^*$. Comme 
par ailleurs on a $K^* \cap \prod_{v\in U^{(1)}} \mathcal{O}_v^*=
k[U]^*$, on obtient une suite exacte:
\begin{equation} \label{gmsuite}
0 \rightarrow \left( \prod_{v\in U^{(1)}} \mathcal{O}_v^*\right) /k[U]^* \rightarrow A(S,\G) \rightarrow \text{Pic}(U) \rightarrow 0.
\end{equation}
Pour chaque $v\in U^{(1)}$, le groupe $\mathcal{O}_v^*$ est divisible
par le lemme de Hensel car le corps r\'esiduel de $\calo_v$ est 
alg\'ebriquement clos de caract\'eristique z\'ero. 

\smallskip

Soit $J$ la jacobienne de la courbe $X$. On a alors une suite exacte 
scind\'ee par le choix de tout point ferm\'e de $X$ (qui est un 
$k$-point puisque $k$ est alg\'ebriquement clos)~:
$$0 \to \pic^0 X \simeq J(k) \to \pic X \to \Z \to 0,$$
et $\pic U$ est le quotient de $\pic X$ par l'image de $\bigoplus_{v \in S} 
\Z. v$. Comme $J(k)$ est divisible (toujours parce que 
$k$ est alg\'ebriquement clos), il en r\'esulte que 
$\overline{\text{Pic}(U)}$ est isomorphe à 
$\mathbb{Z}$ si $S=\emptyset$ et il est trivial sinon.
On en déduit la m\^eme propri\'et\'e pour $\overline{A(S,\G)}$ 
via la suite exacte (\ref{gmsuite}), dont on a vu que le groupe de gauche 
est divisible. 

\end{proof}    

\begin{remarque} \label{quasirem}
Soit $T$ un tore quasi-trivial sur $K$. Le lemme \ref{qt2} impose que $A(S,T)$ est divisible si $S\neq \emptyset$ et que:
$$\overline{A(\emptyset,T)} \cong \mathbb{Z}^{r_T}$$
où $r_T=\text{rg}\, H^0(K,\hat{T})$.
\end{remarque}
    
    \begin{theorem} \label{rangtheo}
    Supposons que $k$ est algébriquement clos. Soit $T$ un $K$-tore.
     \begin{itemize}
    \item[(i)] Le groupe $\overline{A(\emptyset,T)}$ est un groupe abélien de type fini de rang $\text{rg}\,( H^0(K,\hat{T}))$.
       \item[(ii)] Soit $S$ une partie finie non vide de $X^{(1)}$. Alors le groupe ${A(S,T)}$ est divisible.
    \end{itemize}
    \end{theorem}
    
    \begin{proof} Donnons-nous une partie finie $S$ (éventuellement vide) de $X^{(1)}$ et  notons $U=X \setminus S$. D'après le lemme d'Ono (cf. par 
exemple \cite{sansuc}, lemme~1.7), 
il existe un entier $m>0$ et une suite exacte:
    $$0\rightarrow F \rightarrow R_0 \rightarrow T^m  \times R_1 \rightarrow 0$$
    tels que $F$ est un $K$-groupe abélien fini 
et $R_0$ et $R_1$ sont des tores quasi-triviaux. Observons que 
comme $F$ est fini, le groupe $\prod'_{v \in U^{(1)}} F(K_v)$ est simplement 
$\prod_{v \in U^{(1)}} F(K_v)$; d'autre part, le groupe 
$\prod'_{v \in U^{(1)}} H^1(K_v,F)$ est en fait 
$\bigoplus_{v \in U^{(1)}} H^1(K_v,F)$~: en effet, on peut \'etendre $F$ 
en un sch\'ema en groupes fini \'etale ${\cal F}$ au-dessus d'un 
ouvert de Zariski non vide $U_0$ de $X$, apr\`es quoi on a, pour
tout $v \in U_0^{(1)}$, l'\'egalit\'e $H^1(\calo_v,{\cal F})=0$ puisque 
le corps r\'esiduel $k$ de $\calo_v$ est alg\'ebriquement clos.

On obtient alors un diagramme commutatif à lignes exactes: \smallskip \\
    \small
  \centerline{  \xymatrix{
  0 \ar[r] & F(K) \ar[r] \ar[d] & R_0(K) \ar[r]\ar[d] & T(K)^m \times R_1(K) \ar[r]\ar[d] & H^1(K,F) \ar[r] \ar[d]& 0\\
  0 \ar[r] & \prod_{v\in U^{(1)}} F(K_v) \ar[r] & \prod_{v\in U^{(1)}} ' R_0(K_v) \ar[r] & \prod_{v\in U^{(1)}} ' (T(K_v)^m \times R_1(K_v)) \ar[r] & \bigoplus_{v\in U^{(1)}} H^1(K_v,F) \ar[r] & 0
  }}
  \normalsize

\smallskip

  Une chasse au diagramme permet de montrer qu'on a alors (en notant 
pour simplifier encore $H^1(K,F)$ l'image de $H^1(K,F)$ dans 
$\bigoplus_{v \in U^{(1)}} H^1(K_v,F)$) une suite exacte:
  \begin{equation}
  0 \rightarrow N \rightarrow A(S,R_0) \rightarrow A(S,T)^m\times A(S,R_1) \rightarrow \left( \bigoplus_{v\in U^{(1)}} H^1(K_v,F) \right) /H^1(K,F) \rightarrow 0,\label{se3}
\end{equation}

  avec $N$ un groupe abélien d'exposant fini. Comme 
$\overline{A(S,R_0)}$ est un groupe abélien sans torsion d'apr\`es la 
remarque~\ref{quasirem}, l'image de $N$ dans $A(S,R_0)$ est 
contenue dans $A(S,R_0)_{\text{div}}$. Comme $A(S,R_0)$ est isomorphe 
au produit direct de $\overline{A(S,R_0)}$ et de $A(S,R_0)_{\text{div}}$,
on déduit de la suite (\ref{se3}) que l'on a aussi une suite exacte:
  \begin{equation}
  0 \rightarrow \overline{A(S,R_0)} \times D_0 \rightarrow A(S,T)^m\times A(S,R_1) \rightarrow \left( \bigoplus_{v\in U^{(1)}} H^1(K_v,F) \right) /H^1(K,F) \rightarrow 0,\label{se4}
\end{equation}
où $D_0$ est un groupe divisible. Or la suite exacte de Poitou-Tate pour 
un module galoisien fini sur $K$ (cas $d=-1$ de \cite{diego1}, Th. 2.7) 
implique, en posant $F':=\Hom(F,\Q/\Z)$, que:
$$\left( \bigoplus_{v\in X^{(1)}} H^1(K_v,F) \right) /H^1(K,F)\cong F'(K)^D,$$
et donc le groupe:
$$M_U:=\left( \bigoplus_{v\in U^{(1)}} H^1(K_v,F) \right) /H^1(K,F)$$
est fini; de plus, ce groupe est nul si $U \neq X$, c'est-\`a-dire 
si $S \neq \emptyset$ (\cite{colejm}, Prop. 3.2. iii).

En appliquant le lemme \ref{easy}, on obtient alors la suite exacte:
  \begin{equation}
  0 \rightarrow \overline{A(S,R_0)} \rightarrow \overline{A(S,T)}^m\times \overline{A(S,R_1)} \rightarrow M_U \rightarrow 0 .\label{se5}
\end{equation}
\begin{itemize}
\item[(i)] Supposons d'abord que $S =\emptyset$. Dans ce cas, la 
remarque~\ref{quasirem} montre que $\overline{A(\emptyset,R_0)}\cong \mathbb{Z}^{r_0}$ (resp. $\overline{A(\emptyset,R_1)}\cong \mathbb{Z}^{r_1}$) pour $r_0=\text{rg}\, H^0(K,\hat{R_0})$ (resp. $r_1=\text{rg}\, H^0(K,\hat{R_1})$). 
On en déduit que $\overline{A(\emptyset,T)}$ est un groupe abélien de type fini de rang:
$$\frac{r_0-r_1}{m}=\frac{\text{rg}\, H^0(K,\hat{T}^m)}{m} = \text{rg}\, H^0(K,\hat{T}).$$
\item[(ii)] Supposons maintenant que $S$ est non vide. Dans ce cas, 
la remarque~\ref{quasirem} montre que $A(S,R_0)$ et $A(S,R_1)$ sont divisibles. Par conséquent, $\overline{A(S,T)}^m \cong M_U$ est nul et $A(S,T)$ est
divisible.

\end{itemize}

    \end{proof}
    
    \begin{corollary}
   Supposons que $k$ est algébriquement clos. Soit $T$ un $K$-tore. Le groupe $\overline{A(\emptyset,T)}$ est fini si, et seulement si, $T$ est anisotrope.
    \end{corollary}

\subsection{Une application aux groupes lin\'eaires connexes} 

On suppose toujours que $K$ est le corps des fonctions d'une courbe projective
et lisse $X$ sur un corps alg\'ebriquement clos de caract\'eristique 
z\'ero $k$. Soit $S$ un ensemble fini de points ferm\'es de $X$.
Dans ce paragraphe, on \'etend le th\'eor\`eme~\ref{rangtheo}, ii) \`a un 
$K$-groupe alg\'ebrique lin\'eaire connexe $G$ quelconque. On note 
$G(S,\A_K)$ le produit restreint des $G(K_v)$ pour $v \not \in S$ 
et ${G(K)}_{\adh}$ l'adh\'erence de $G(K)$ dans $G(S,\A_K)$, puis 
$A(S,G)$ l'ensemble quotient $G(S,\A_K)/{G(K)}_{\adh}$, qui est donc 
le d\'efaut d'approximation forte en dehors de $S$.

\begin{theorem} \label{lineairetheo}
Soit $S$ un ensemble fini non vide de points ferm\'es de $X$. Soit $G$ un 
$K$-groupe lin\'eaire connexe. Alors ${G(K)}_{\adh}$ est un 
sous-groupe normal de $G(S,\A_K)$ et le quotient $A(S,G)$ est un groupe 
ab\'elien divisible.
\end{theorem}

\dem On va d\'eduire ce r\'esultat du th\'eor\`eme~\ref{rangtheo}, ii) et du 
th\'eor\`eme principal de \cite{colejm}, en suivant une m\'ethode un 
peu similaire \`a celle de \cite{sansuc}, Th. 3.3. 

\smallskip 

On commence par observer que comme le corps $K$ et
les compl\'et\'es $K_v$ sont $C_1$, le th\'eor\`eme de Steinberg
(\cite{cogal}, \S III.2.2., Th. 1) assure
que $H^1(K,L)=H^1(K_v,L)=1$ pour tout groupe lin\'eaire connexe 
$L$. Il en r\'esulte que toute suite exacte de $K$-groupes lin\'eaires 
connexes 
$$1 \to L_1 \to L_2 \to L_3 \to 1$$ 
induit une suite exacte de groupes
$$1 \to L_1(K) \to L_2(K) \to L_3(K) \to 1 , $$
et de m\^eme si on remplace $K$ par $K_v$. On en d\'eduit imm\'ediatement 
par d\'evissage que si $L_1$ et $L_3$ satisfont l'approximation forte en 
dehors de $S$, il en va de m\^eme de $L_2$. Or, le groupe additif 
${\bf G}_a$ v\'erifie l'approximation forte en dehors de $S$ parce que 
$S$ est non vide, d'apr\`es le th\'eor\`eme d'approximation forte 
pour les anneaux de Dedekind (\cite{cohn}, Th. 10.5.10) appliqu\'e 
\`a l'anneau des fonctions r\'eguli\`eres sur la courbe affine $U:=X-S$. 
Comme tout groupe connexe unipotent en caract\'eristique z\'ero s'obtient 
\`a partir du groupe trivial via des extensions successives par ${\bf G}_a$,
on en d\'eduit que tout $K$-groupe unipotent connexe satisfait 
l'approximation forte en dehors de $S$. 

\smallskip
%% modifié ici.
Si $U$ d\'esigne le radical unipotent de $G$, le quotient 
$H:=G/U$ est r\'eductif; de plus, en tant que $K$-vari\'et\'e, 
$G$ est isomorphe au produit $U \times_K H$ par le th\'eor\`eme de 
Mostow. Il en r\'esulte 
que $G(K)_{\adh}$ est l'image r\'eciproque de $H(K)_{\adh}$ par la 
surjection canonique $G(S,\A_K) \to H(S,\A_K)$. Pour d\'emontrer le 
th\'eor\`eme, on peut donc supposer que $G$ est r\'eductif.

\smallskip

On sait alors que $G$ s'ins\`ere dans une suite exacte de $K$-groupes 
alg\'ebriques
$$1 \to L \to G \to T \to 1,$$
o\`u $L$ est semi-simple et $T$ est un tore. 
Le th\'eor\`eme de Steinberg fournit encore
un diagramme commutatif \`a lignes exactes de groupes topologiques
$$
\begin{CD}
1 @>>> L(K) @>>> G(K) @>{\varphi}>> T(K) @>>> 1 \cr
&& @VVV @VVV @VV{\theta}V \cr
1 @>>> \prod_{v \not \in S} ' L(K_v) @>>> \prod_{v \not \in S} ' G(K_v) 
@>{u}>> \prod_{v \not \in S} ' T(K_v) @>>> 1 .
\end{CD}
$$
Soit $I=u^{-1}(\im \theta)$. Comme $T(K)$ est ferm\'e dans 
$\prod_{v \not \in S} ' T(K_v)$ d'apr\`es le th\'eor\`eme~\ref{gentheo},
le sous-groupe $I$ est ferm\'e dans $\prod_{v \not \in S} ' G(K_v)$, et 
il est \'egalement normal car $\prod_{v \not \in S} ' T(K_v)$ est ab\'elien.
Ainsi $I$ contient l'adh\'erence $G(K)_{\adh}$. On obtient 
un diagramme commutatif \`a lignes exactes
$$
\begin{CD}
1 @>>> L(K) @>>> G(K) @>{\varphi}>> T(K) @>>> 1 \cr
&& @V{i}VV @VVV @VV{\theta}V \cr
1 @>>> \prod_{v \not \in S} ' L(K_v) @>>> I
@>{u}>> \im \theta @>>> 1 ,
\end{CD}
$$
o\`u la fl\`eche $\theta$ est maintenant surjective. Comme 
$L(K)$ est dense dans $\prod_{v \not \in S} ' L(K_v)$ d'apr\`es 
le Th. 3.4. de \cite{colejm}, une chasse au diagramme imm\'ediate 
donne alors que $G(K)_{\adh}=I$. Ainsi 
$$A(S,G)=(\prod_{v \not \in S} ' G(K_v))/I \cong (\prod_{v \not \in S} ' T(K_v))
/T(K)=A(S,T)$$
est bien un groupe ab\'elien divisible, ce qui ach\`eve la preuve.

\enddem

\begin{remarque}
%% remarque précisée.
Dans le cas o\`u $S=\emptyset$ et $G=G_0 \times_k K$ (o\`u $G_0$ est 
un $k$-groupe lin\'eaire), on peut prendre $\mathcal G=G \times_k X$ 
comme mod\`ele de $G$ au-dessus de $X$.
Alors, l'intersection 
$$G(K) \cap \prod_{v \in X^{(1)}} {\mathcal G}(\calo_v) ={\mathcal G}(X)=
\Hom_k(X,G_0)$$ 
est r\'eduite \`a $G_0(k)$ car $X$ est projective et $G_0$ affine 
sur $k$. On en d\'eduit imm\'ediatement que $G(K)$ est discret 
dans $G(\emptyset,\A_K)$ (ceci pour $k$ quelconque). 
%% peut-etre vrai meme si $G$ n'est pas isotrivial ?
Il ne semble pas y avoir de bonne description de l'ensemble quotient  
$G(\emptyset,\A_K)/G(K)$ en g\'en\'eral (m\^eme si $G$ est semi-simple 
et simplement connexe), tout comme dans le cas d'un corps de nombres 
ou d'un corps de fonctions sur un corps fini o\`u il faut toujours enlever 
au moins une place pour obtenir des \'enonc\'es d'approximation forte. 
Nous allons voir au paragraphe suivant que la situation est meilleure 
si on se limite aux tores.
\end{remarque}

        \subsection{Suite de Poitou-Tate}
   
Soit $T$ un $K$-tore. Notre but ici est de traiter le cas "$d=-1$" du
th\'eor\`eme~3.20 de \cite{diego1}, et d'en d\'eduire une suite exacte analogue
\`a celle du corps de classes global reliant $A(\emptyset,T)$ \`a un groupe
de cohomologie galoisienne sur $K$ associ\'e \`a $T$. Dans loc. cit., o\`u 
le corps $k$ est un corps $d$-local, 
c'est le complexe motivique $\Z(d)$ qui appara\^{\i}t. Suivant une suggestion 
de B. Kahn, que nous remercions pour son aide, c'est $\mathbb{Q}/\mathbb{Z}(-1)$ qui va jouer ce r\^ole dans notre cas $d=-1$. Plus pr\'ecis\'ement, on 
pose, en suivant \cite{kahn}, d\'efinition~4.1 ~:
 
$$\tilde{T} = \hat{T} \otimes \mathbb{Q}/\mathbb{Z}(-1).$$ 

Rappelons que $\G=\Z(1)[1]$ dans la cat\'egorie d\'eriv\'ee des 
modules galoisiens sur $K$. On a alors (\cite{kahn}, section 5) 
dans cette cat\'egorie un accouplement naturel:

   \begin{equation}
   T \otimes^{\mathbb L} \tilde{T} \rightarrow \mathbb{Z}(1)[1] \otimes^{\mathbb L} \mathbb{Q}/\mathbb{Z}(-1) \rightarrow \mathbb{Z}[2].
   \end{equation}
   Pour $v \in X^{(1)}$, on obtient alors un accouplement naturel en cohomologie:
   \begin{equation}\label{accloc}
   H^0(K_v,T) \times H^0(K_v,\tilde{T}) \rightarrow H^2(K_v,\mathbb{Z}) \cong \mathbb{Q}/\mathbb{Z}.
      \end{equation}

Notons que $\tilde T$ est de torsion, avec $_n \tilde T=\hat T \otimes 
\Z/n \Z(-1)$ pour tout $n>0$.  
   
   \begin{lemma} \label{localkalg}
   Soit $v\in X^{(1)}$. L'accouplement (\ref{accloc}) 
induit un accouplement parfait entre un groupe profini et un 
groupe discret de torsion:
   $$H^0(K_v,T)_{\wedge} \times H^0(K_v,\tilde{T}) \rightarrow \mathbb{Q}/\mathbb{Z}.$$
   \end{lemma}
   
   \begin{proof}

Soit $n>0$. On \'ecrit la suite exacte de Kummer
$$0 \to _n T \to T \stackrel{.n}{\to} T \to 0.$$
On a $H^1(K_v,T)=0$ (ce groupe est d'exposant fini par Hilbert 90 et 
$H^1(K_v,T)/n$ s'injecte dans $H^2(K_v,_n T)$ via la suite de Kummer, avec 
de plus $K_v$ de dimension 
cohomologique $1$ comme corps 
complet pour une valuation discr\`ete de corps r\'esiduel alg\'ebriquement 
clos). Ainsi $H^0(K_v,T)/n \cong H^1(K_v,{_n}T)$. Par \cite{adt}, partie I,
exemple 1.10,
le groupe $H^1(K_v,{_n}T)$ est fini et il est dual de 
$H^0(K_v, \hat T \otimes \Z/n \Z(-1))$, ou encore de 
$H^0(K_v, _n \tilde T)$.

\smallskip

Finalement, pour chaque $n>0$, on a un isomorphisme naturel:
   $$H^0(K_v,T)/n \cong ({_n}H^0(K_v,\tilde{T}))^D.$$
   En passant à la limite projective sur $n$ et en notant que $\tilde T$ 
est de torsion, on obtient un 
isomorphisme~ :
   $$H^0(K_v,T)_{\wedge} \cong H^0(K_v,\tilde{T})^D.$$

Notons que $H^0(K_v,T)_{\wedge}=H^0(K_v,T)^{\wedge}$ est bien profini
puisque chaque $H^0(K_v,T)/n$ est fini.
   \end{proof}
   
Soit $\mathcal{T}$ un tore sur un ouvert non vide $U$ de $X$ étendant $T$. Soit $\tilde{\mathcal{T}} =    \hat{\mathcal{T}} \otimes \mathbb{Q}/\mathbb{Z}(-1)$.

  Posons:
\begin{gather*}   
   \mathbb{P}^0(\tilde{T}) := \prod_{v\in X^{(1)}} \tilde{T}(K_v),\\
    \mathbb{P}^1(\tilde{T}) := \bigoplus_{v\in X^{(1)}} H^1(K_v,\tilde{T}).
   \end{gather*}
Pour chaque $n>0$, munissons  
${_n}\mathbb{P}^0(\tilde{T}) := \prod_{v\in X^{(1)}} {_n}\tilde{T}(K_v)$ 
de la topologie produit (obtenue \`a partir 
des topologies discr\`etes sur chaque groupe fini ${_n}\tilde{T}(K_v)$), 
qui en fait un espace compact.
 Munissons ensuite $\mathbb{P}^0(\tilde{T})_{\text{tors}} = \varinjlim_n {_n}\mathbb{P}^0(\tilde{T})$ de la topologie limite inductive. Noter 
que $\mathbb{P}^1(\tilde{T})$ est aussi le produit restreint 
$\prod_{v\in X^{(1)}} ' H^1(K_v,\tilde{T})$~: en effet on a 
$H^1(\calo_v, \tilde{\mathcal T})=0$ pour $v \in U^{(1)}$ puisque le corps 
r\'esiduel de $\calo_v$ est alg\'ebriquement clos. 

\begin{remarque}
Contrairement au cas o\`u $K$ est un corps de nombres, on n'a 
pas en g\'en\'eral $\Sha^1(K,\Q/\Z)=0$, faute d'avoir l'analogue du 
th\'eor\`eme de \v Cebotarev. Plus pr\'ecis\'ement, $\Sha^1(K,\Q/\Z)$ 
s'identifie au dual du groupe fondamental \'etale de $X$ (ou de son 
ab\'elianis\'e), il est donc nul seulement quand $X$ est la droite 
projective, et est infini en g\'en\'eral. 
\end{remarque}

\begin{lemma} \label{firstlem}
On a une suite exacte :
$$0 \rightarrow H^0(K,\tilde{T}) \rightarrow \mathbb{P}^0(\tilde{T})_{\text{tors}} \rightarrow (H^0(K,T)_{\wedge})^D \rightarrow \Sha^1(K,\tilde{T}) \rightarrow 0.$$
\end{lemma}  

\begin{proof}
On a un diagramme commutatif: \smallskip \\
\centerline{\xymatrix{
0 \ar[r] & \varinjlim_n H^0(K,{_n}\tilde{T}) \ar[r]\ar[d]^{\cong} &  \varinjlim_n \mathbb{P}^0({_n}\tilde{T})\ar[r]\ar[d]^{\cong} &  
\varinjlim_n (H^1(K,{_n}T)^D) \ar[r]\ar[d]^{\cong} &  \varinjlim_n H^1(K,{_n}\tilde{T}) \ar[r]\ar[d]^{\cong} &  \varinjlim_n \mathbb{P}^1({_n}\tilde{T})\ar[d]^{\cong} \\
0 \ar[r] & H^0(K,\tilde{T}) \ar[r] & \mathbb{P}^0(\tilde{T})_{\text{tors}}\ar[r] & (H^0(K,T)_{\wedge})^D \ar[r] & H^1(K,\tilde{T}) \ar[r] & \mathbb{P}^1(\tilde{T}) 
}}

\smallskip

Montrons d'abord que tous les morphismes verticaux sont des 
isomorphismes. Le premier est clair 
parce que $\tilde T$ est de torsion, et pour le deuxi\`eme c'est imm\'ediat. 
Le troisi\`eme vient des isomorphismes $H^0(K,T)/n \cong H^1(K,_n T)$, lesquels 
r\'esultent de la suite de Kummer et de ce que $H^1(K,T)=0$ puisque $K$ 
est un corps $C_1$ par le th\'eor\`eme de Tsen. Le quatri\`eme vient 
de ce que $\tilde T$ est de torsion. Enfin, pour le dernier, 
on utilise la commutation de $\varinjlim$ avec $\oplus$ et avec la 
cohomologie galoisienne.

\smallskip

Maintenant, la première ligne est exacte en passant \`a la limite sur 
la suite de Poitou-Tate pour les modules finis (\cite{diego1}, Th.~2.7 dans 
le cas $d=-1$).
Il en est donc de même de la deuxième, d'où le lemme.
\end{proof}
  
  \begin{lemma}
  La dualité locale induit un isomorphisme:
  $$\mathbb{P}^0(T)_{\wedge} \cong (\mathbb{P}^0(\tilde{T})_{\rm tors})^D.$$
  \end{lemma}
  
  \begin{proof} On observe que pour $v \not \in U^{(1)}$, le groupe 
$H^0(\calo_v,{\mathcal T})$ est divisible via le lemme de Hensel, 
puisque la r\'eduction modulo $v$ de ${\mathcal T}$ est un tore sur 
un corps alg\'ebriquement clos de caract\'eristique z\'ero. Il en 
r\'esulte que pour tout $n>0$ le groupe $\mathbb{P}^0(T)/n$ s'identifie 
\`a $\bigoplus_{v \in X^{(1)}} H^0(K_v,T)/n$. En utilisant alors le 
lemme~\ref{localkalg}, on obtient~:
$$\mathbb{P}^0(T)/n \cong ({_n}\mathbb{P}^0(\tilde{T}))^D.$$
Il suffit alors de passer à la limite projective sur $n$.
  \end{proof}

   \begin{theorem} \label{poitoukalg}
   On a une suite exacte (de type Poitou-Tate)~:
   $$0 \rightarrow \Sha^{1}(K,\tilde{T})^D \rightarrow H^0(K,T)_{\wedge} \rightarrow \mathbb{P}^0(T)_{\wedge} \rightarrow H^0(K,\tilde{T})^D \rightarrow 0.$$
   \end{theorem}

   \begin{proof}
   D'apr\`es le lemme~\ref{firstlem}, on a une suite exacte~:
   $$0 \rightarrow H^0(K,\tilde{T}) \rightarrow \mathbb{P}^0(\tilde{T})_{\text{tors}} \rightarrow (H^0(K,T)_{\wedge})^D \rightarrow \Sha^1(K,\tilde{T}) \rightarrow 0$$
   telle que:

\smallskip

   \begin{itemize}
\item[$\bullet$]   l'image du morphisme $H^0(K,\tilde{T}) \rightarrow \mathbb{P}^0(\tilde{T})_{\text{tors}}$ est discrète car $H^0(K,\tilde{T})$ est de torsion de type cofini; 
\item[$\bullet$] le groupe $\Sha^1(K,\tilde{T})$ est discret;
\item[$\bullet$] le groupe localement compact 
$\mathbb{P}^0(\tilde{T})_{\text{tors}}$ est réunion dénombrable 
d'espaces compacts. Le groupe $H^0(K,T)_{\wedge}$ est limite projective
de groupes discrets de torsion, donc le groupe $(H^0(K,T)_{\wedge})^D$ 
est localement compact (limite inductive de groupes profinis) et son dual 
est $H^0(K,T)_{\wedge}$; 
\item[$\bullet$] l'image du morphisme 
$\mathbb{P}^0(\tilde{T})_{\text{tors}} \rightarrow (H^0(K,T)_{\wedge})^D$ 
est fermée: ce morphisme est donc strict (\cite{hewitt}, Th. 5.29).
\end{itemize}
On en déduit (en utilisant \cite{dhsza2}, Lemme~2.4) que la suite duale:
$$0 \rightarrow \Sha^{1}(K,\tilde{T})^D \rightarrow H^0(K,T)_{\wedge} \rightarrow \mathbb{P}^0(T)_{\wedge} \rightarrow H^0(K,\tilde{T})^D \rightarrow 0$$
est exacte.

   \end{proof}

\begin{remarque}
La fl\`eche
$\mathbb{P}^0(T)_{\wedge} \to H^0(K,\tilde{T})^D $ est ici induite
par une fl\`eche ``de r\'eciprocit\'e'' $r : \mathbb{P}^0(T) \to 
H^0(K,\tilde{T})^D$ (dont le noyau contient l'image de $H^0(K,T)$), 
l'accouplement correspondant
\begin{equation} \label{reciproc}
\mathbb{P}^0(T) \times H^0(K,\tilde{T}) \to \Q/\Z
\end{equation}
\'etant d\'efini pour tout $(t_v) \in \mathbb{P}^0(T)$ et tout
$\tilde t \in H^0(K,\tilde{T})$ par la formule
$$((t_v),\tilde t)=\sum_{v \in X^{(1)}} (t_v,\tilde t_v)_v,$$
o\`u $(,)_v$ est l'accouplement local en $v$ et $\tilde t_v$ est
l'image de $\tilde t$ dans $H^0(K_v,\tilde T)$.
\end{remarque}

   \begin{theorem} \label{kalgclos}
   Soit $\rho = \text{rg}(H^0(K,\hat{T}))$. Rappelons que $A(\emptyset,T)$ désigne le groupe $\mathbb{P}^0(T)/T(K)$. 
\begin{itemize}   
\item[(i)] On a un isomorphisme naturel $(\overline{A(\emptyset,T)})_{\text{tors}} \cong H^1(K,\hat{T})^D$. 
\item[(ii)]   On a une suite exacte:
   $$0 \rightarrow \overline{A(\emptyset,T)} \rightarrow H^0(K,\tilde{T})^D \rightarrow (\hat{\mathbb{Z}}/\mathbb{Z})^{\rho} \rightarrow 0,$$
   telle que l'image $I$ du morphisme $r : A(\emptyset,T) 
\rightarrow H^0(K,\tilde{T})^D$ soit dense et 
vérifie $I\otimes \hat{\mathbb{Z}} \cong  H^0(K,\tilde{T})^D$.
   \end{itemize}
   \end{theorem}

En particulier, ce th\'eor\`eme dit que le noyau de la fl\`eche de 
r\'eciprocit\'e est divisible modulo $T(K)$, et il d\'ecrit son conoyau.

\begin{proof}
\begin{itemize}
\item[$\bullet$] D'apr\`es le lemme~\ref{completion} b), le 
noyau de $r : A(\emptyset,T) \rightarrow H^0(K,\tilde{T})^D$ est 
$A(\emptyset,T)_{\infty-\div}$ puisqu'on a la suite exacte:
$$H^0(K,T)_{\wedge} \rightarrow \mathbb{P}^0(T)_{\wedge} \rightarrow H^0(K,\tilde{T})^D.$$ Il coïncide avec le sous-groupe divisible maximal de $A(\emptyset,T)$ car $\overline{A(\emptyset,T)}$ est de type fini. Cela montre l'injectivité de $ \overline{A(\emptyset,T)} \rightarrow H^0(K,\tilde{T})^D$. De plus, 
l'image de $\overline{A(\emptyset,T)}$ dans le groupe profini 
$H^0(K,\tilde{T})^D$ est dense d'apr\`es le lemme~\ref{completion} a) et le 
th\'eor\`eme~\ref{poitoukalg}.

\item[$\bullet$] Par ailleurs, on a une suite exacte scindée:
$$ 0  \rightarrow H^0(K,\hat{T}) \otimes \mathbb{Q}/\mathbb{Z} \rightarrow H^0(K,\tilde{T}) \rightarrow H^1(K,\hat{T}) \rightarrow 0$$
où $H^1(K,\hat{T})$ est fini et où $H^0(K,\hat{T})\cong \mathbb{Z}^{\rho}$. On en déduit que $H^0(K,\tilde{T})^D$ est isomorphe à $\hat{\mathbb{Z}}^{\rho} \oplus H^1(K,\hat{T})^D$ avec $H^1(K,\hat{T})^D$ fini. Toutes les assertions 
restantes du théorème découlent alors du lemme qui suit, puisque l'image de $A(\emptyset,T) \rightarrow H^0(K,\tilde{T})^D$ est un groupe abélien de type fini de rang $\rho$ qui est dense dans $H^0(K,\tilde{T})^D$.
\end{itemize}

\begin{lemma}
Soit $G$ un groupe topologique de la forme 
$\hat{\mathbb{Z}}^{\rho}\oplus F$ pour un certain entier naturel $\rho$ 
et un certain groupe abélien fini $F$. Soit $H$ un sous-groupe dense de $G$. 
Supposons que $H$ est de type fini et de rang $\rho$. Alors:
\begin{itemize}
\item[(i)] le groupe $H$ contient $F$,
\item[(ii)] on a l'égalité $H \otimes \hat{\mathbb{Z}}= G$,
\item[(iii)] on a un isomorphisme $G/H \cong (\hat{\mathbb{Z}}/\mathbb{Z})^{\rho}$.
 \end{itemize}
\end{lemma}

On peut supposer que $H$ est sans torsion, quitte \`a remplacer $G$ et 
$H$ par leur quotient par $H_{\rm tors}$. 
Soit  $(e_1,...,e_{\rho})$ une $\mathbb{Z}$-base de $H$. 
Soit $p$ un nombre premier. 
L'image de $H$ dans $\mathbb{Z}_p^{\rho} \oplus F\{p\}$ doit être dense. 
Par conséquent, l'image de $(e_1,...,e_{\rho})$ 
doit être $\mathbb{Z}_p$-libre (sinon elle engendrerait un sous-groupe 
de rang $<\rho$ dans $G$),  d'o\`u on déduit que $F\{p\}=0$. 
Ceci étant vrai pour tout $p$, on obtient que $F=0$, ce qui prouve (i). 
Alors, l'image de $(e_1,...,e_{\rho})$ dans $\mathbb{Z}_p^{\rho}$ 
doit être une $\mathbb{Z}_p$-base de $\mathbb{Z}_p^{\rho}$ quel que soit $p$, 
d'o\`u on voit que $(e_1,...,e_{\rho})$ est une 
$\hat{\mathbb{Z}}$-base de $G=\hat{\mathbb{Z}}^{\rho}$, 
ce qui prouve (ii) et (iii).

\end{proof}

\begin{corollary}
On a une suite exacte:
$$0\rightarrow H^1(K,\hat{T})^D\rightarrow \overline{A(\emptyset,T)} \rightarrow \mathbb{Z}^{\rho} \rightarrow 0 ,$$
 dans laquelle le groupe $H^1(K,\hat{T})$ est fini. En particulier, le groupe $A(\emptyset,T)$ est divisible si, et seulement si, $T$ est anisotrope et $H^1(K,\hat{T})=0$.
\end{corollary}

\begin{remarque}
Le groupe $H^0(K,\tilde T)$ peut \'egalement \^etre interpr\'et\'e de la 
mani\`ere suivante. Posons $\ov T=T \times_K \ov K$. 
Comme $\pic \ov T=0$, le module galoisien $H^1(\ov T, \mu_n)$ 
est isomorphe \`a $\ov K[T]^*/\ov K[T]^{*^n} \cong \hat T/n$, et 
le module galoisien
$H^1(\ov T,\Z/n) \cong \hat T \otimes \Z/n(-1)$ s'identifie alors 
\`a $_n \tilde T$. 
La suite spectrale de Hochschild-Serre jointe au fait que $K$ est de dimension
cohomologique $1$ permet alors d'identifier $H^1(T,\Z/n)/H^1(K,\Z/n)$ 
avec $H^0(K, H^1(\ov T,\Z/n)) \cong _n H^0(K,\tilde T)$. On peut alors 
voir l'obstruction pour un point ad\'elique $(P_v)_{v \in X^{(1)}}$ \`a 
provenir d'un point rationnel comme une obstruction de r\'eciprocit\'e, 
analogue \`a la classique obstruction de Brauer-Manin sur les corps 
de nombres~: 
le point $(P_v)$ est dans l'image de $T(K)$ modulo divisible si et seulement 
si on a 
$$\sum_{v \in X^{(1)}} \alpha(P_v)=0$$
pour tout $\alpha \in H^1(T,\Z/n)$, o\`u $\alpha(P_v) \in H^1(K_v,\Z/n) 
\cong \Z/n$.
\end{remarque}

\subsection{Le sous-groupe de torsion de $A(\emptyset,T)$}\label{ctors}

Le but de ce paragraphe est de d\'ecrire plus pr\'ecis\'ement 
la partie de torsion de $A(\emptyset,T)_{\div}$, ce qui d\'eterminera 
compl\`etement la structure de $A(\emptyset,T)$ \`a un sous-groupe 
uniquement divisible pr\`es. 

\smallskip

Plus pr\'ecis\'ement, nous allons maintenant construire un 
accouplement naturel:
\begin{equation}\label{acc}
 A(\emptyset,T)_{\text{tors}}  \times \varprojlim_n H^1(K,{_n}\tilde{T}) \rightarrow \mathbb{Q}/\mathbb{Z}.
 \end{equation}

Pour ce faire, nous commençons par introduire quelques notations et rappels:

\begin{notation}
Soit $U_0$ un ouvert non vide \emph{strict} de $X$ tel que $T$ s'étende en un tore $\mathcal{T}$ sur $U_0$. Pour chaque ouvert non vide $U$ de $U_0$, on note $j_U: U \hookrightarrow X$ l'immersion ouverte et on définit la cohomologie à support compact de $\mathcal{T}$ par $H^i_c(U,\mathcal{T}) := H^i(X,{(j_U)}_!\mathcal{T})$. Si on note $K_v^h$ l'hensélisé de $K$ par rapport à $v$ pour chaque $v\in X^{(1)}$, on rappelle la suite exacte (\cite{dhsza1}, Prop. 3.1):
\begin{equation}\label{localisation}
... \rightarrow H^r_c(U,\mathcal{T}) \rightarrow H^i(U,\mathcal{T}) \rightarrow \bigoplus_{v\in X \setminus U} H^i(K_v^h,T) \rightarrow ...
\end{equation}
Lorsque $V$ est un ouvert non vide de $U$, on a un diagramme commutatif (\cite{dhct}, Prop. 4.3): 
\begin{equation}\label{3arrow}
\xymatrix{
\bigoplus_{v\in X \setminus U} H^i(K_v^h,T) \ar[r]\ar[d] & H^{i+1}_c(U,\mathcal{T})\\
\bigoplus_{v\in X \setminus V} H^i(K_v^h,T) \ar[r] & H^{i+1}_c(V,\mathcal{T}), \ar[u]
}
\end{equation}
o\`u la fl\`eche verticale de gauche est donn\'ee par 
$(x_v)_{v \in X \setminus U} \mapsto ((x_v),0,...,0)$.
\end{notation}

  Donnons-nous maintenant un ouvert $U$ non vide de $U_0$. Soient $v\in X^{(1)}$ et $V=U\setminus \{v\}$. On peut alors définir un morphisme $T(K_v^h)\rightarrow H^1_c(U,\mathcal{T})$ par composition:
$$f_v: T(K_v^h) \rightarrow H^1_c(V,\mathcal{T})\rightarrow H^1_c(U,\mathcal{T}).$$
Ici, le morphisme $T(K_v^h) \rightarrow H^1_c(V,\mathcal{T})$ vient de la suite (\ref{localisation})  et donc de l'identification:
$$T(K_v^h) \cong H^1_v(\mathcal{O}_v^h,{j_V}_!\mathcal{T}) \cong H^1_v(X,{j_V}_!\mathcal{T})$$
et du morphisme naturel $H^1_v(X,{j_V}_!\mathcal{T}) \rightarrow H^1_c(V,\mathcal{T})$. \\

Supposons maintenant que $v\in U^{(1)}$. 
Dans ce cas, on a
$H^0(\calo_v^h,{j_U}_! \mathcal{T})=\mathcal{T}(\mathcal{O}_v^h)$ et
$H^0(K_v^h, {j_U}_! \mathcal{T})=T(K_v^h)$. De plus, on a
$H^1_v(X, {j_U}_! \mathcal{T})=H^1_v(\calo_v^h,{j_U}_! \mathcal{T})$ par
excision. La suite de localisation pour l'immersion ouverte
$\spec K_v^h \hookrightarrow \spec \calo_v^h$ donne alors une
suite exacte
$$\mathcal{T}(\mathcal{O}_v^h) \to T(K_v^h) \to H^1_v(X,{j_U}_!\mathcal{T}) ,$$
d'o\`u on tire un diagramme commutatif:\\

\centerline{\xymatrix{
T(K_v^h) \ar[r]^-{\cong} \ar[d] & H^1_v(X,{j_V}_!\mathcal{T}) \ar[r]\ar[d]& H^1_c(V,\mathcal{T})\ar[d] \\
T(K_v^h)/\mathcal{T}(\mathcal{O}_v^h) \ar@{^{(}->}[r] &H^1_v(X,{j_U}_!\mathcal{T})\ar[r] & H^1_c(U,\mathcal{T}).
}}
 On en déduit que $f_v(\mathcal{T}(\mathcal{O}_v^h))=0$.\\

En sommant tous les $f_v$, on définit donc un morphisme :
$$\prod'_{v\in X^{(1)}} T(K_v^h) \rightarrow H^1_c(U,\mathcal{T}).$$
La commutativité du diagramme (\ref{3arrow}) permet alors de passer à la limite projective sur $U$ pour obtenir un morphisme:
$$\prod'_{v\in X^{(1)}} T(K_v^h) \rightarrow \varprojlim_U H^1_c(U,\mathcal{T}).$$
Remarquons que, si $x$ est un élément de $T(K) \subseteq \prod'_{v\in X^{(1)}} T(K_v^h)$, il appartient à $\mathcal{T}(U)$ pour un certain ouvert non vide $U$ de $U_0$. On déduit alors de la suite (\ref{localisation}) que son image dans $H^1_c(V,\mathcal{T})$ est nulle pour chaque ouvert non vide $V$ de $U$. On obtient donc un morphisme:
$$\frac{\prod'_{v\in X^{(1)}} T(K_v^h)}{T(K)} \rightarrow \varprojlim_U H^1_c(U,\mathcal{T}),$$
d'où finalement un morphisme :
$$ \left( \frac{\prod'_{v\in X^{(1)}} T(K_v^h)}{T(K)}\right)_{{\rm tors}} \rightarrow \varinjlim_n \varprojlim_U {_n}H^1_c(U,\mathcal{T}) \cong \varinjlim_n \varprojlim_U H^1_c(U,{_n}\mathcal{T})$$
via l'identification ${_n}H^1_c(U,\mathcal{T}) \cong H^1_c(U,{_n}\mathcal{T})$ (le groupe $H^0_c(U,\mathcal{T})=\ker [{\mathcal T}(U) \to \oplus_{v \not \in U} 
T(K_v)]$ est nul car $U \neq X$). 

\begin{lemma}
On a l'égalité:
$$\left( \frac{\prod'_{v\in X^{(1)}} T(K_v^h)}{T(K)}\right)_{{\rm tors}} = A(\emptyset,T)_{{\rm tors}}.$$
\end{lemma}

\begin{proof}
On a une inclusion évidente $\left( \frac{\prod'_{v\in X^{(1)}} T(K_v^h)}{T(K)}\right)_{\text{tors}} \subseteq A(\emptyset,T)_{\text{tors}}$. Soit donc $(x_v)_{v\in X^{(1)}}\in A(\emptyset,T)_{\text{tors}}$. Soit $(\tilde{x}_v)$ un relèvement de $(x_v)$ à $ \mathbb{P}^0(T)$. Soient $n\geq 1$ et $x\in T(K)$ tels que $n\tilde{x}_v=x$ pour chaque $v\in X^{(1)}$. Comme $K_v^h$ est algébriquement clos dans $K_v$, on en déduit que $\tilde{x}_v \in T(K_v^h)$ pour chaque $v$. Autrement dit, $(x_v) \in \left( \frac{\prod'_{v\in X^{(1)}} T(K_v^h)}{T(K)}\right)_{\text{tors}}$. 
\end{proof}

On obtient donc un morphisme:
$$ A(\emptyset,T)_{\text{tors}} \rightarrow \varinjlim_n \varprojlim_U {_n}H^1_c(U,\mathcal{T}) \cong \varinjlim_n \varprojlim_U H^1_c(U,{_n}\mathcal{T}).$$
Pour chaque $n>0$, munissons ${_n}A(\emptyset,T)$ (resp. ${_n}\mathbb{P}^0(T)$) de la topologie induite par $A(\emptyset,T)$ (resp. $\mathbb{P}^0(T)$) et $H^1(K,{_n}\tilde{T})$ de la topologie discrète. Munissons ensuite le 
groupe $A(\emptyset,T)_{\rm tors} = \varinjlim_n {_n}A(\emptyset,T)$ (resp. $\mathbb{P}^0(T)_{\rm tors} = \varinjlim_n {_n}\mathbb{P}^0(T)$) de la topologie limite inductive et $\varprojlim_n H^1(K,{_n}\tilde{T})$ de la topologie 
limite projective.
En exploitant la dualité de Poincaré, on a un isomorphisme~:

$$ \left( \varinjlim_n \varprojlim_U H^1_c(U,{_n}\mathcal{T})\right)^D  \cong \varprojlim_n \varinjlim_U H^1(U,{_n}\tilde{\mathcal{T}}) \cong \varprojlim_n H^1(K,{_n}\tilde{T})$$
puisque ${_n}\tilde{\mathcal{T}} \cong \hat{\mathcal{T}} \otimes \mathbb{Z}/n\mathbb{Z}(-1)$,
d'où l'accouplement de groupes topologiques (\ref{acc}).

\begin{lemma}\label{ators}
On a une suite exacte :
$$0 \rightarrow \frac{\mathbb{P}^0(T)_{{\rm tors}}}{T(K)_{{\rm tors}}}\rightarrow  A(\emptyset,T)_{\rm tors}\rightarrow \varinjlim_n \Sha^1(K,{_n}T) \rightarrow 0.$$
\end{lemma}

\begin{proof}
Pour $n\geq 1$, soit: $$h_n: {_n}A(\emptyset,T) \rightarrow \text{Ker} \left( T(K)/n \rightarrow \prod_{v\in X^{(1)}} T(K_v)/n \right)$$
 qui à une famille $(x_v) \in \mathbb{P}^0(T)$ telle qu'il existe $x\in T(K)$ vérifiant $nx_v=x$ pour tout $v$ associe $x$. Le morphisme $h_n$ est surjectif et son noyau est $\frac{{_n}\mathbb{P}^0(T)}{{_n}T(K)}$. On obtient donc une suite exacte :
 $$0 \rightarrow \frac{{_n}\mathbb{P}^0(T)}{{_n}T(K)} \rightarrow {_n}A(\emptyset,T) \rightarrow \text{Ker} \left( T(K)/n \rightarrow \prod_{v\in X^{(1)}} T(K_v)/n \right) \rightarrow 0.$$
Via les identifications $T(K)/n \cong H^1(K,{_n}T)$ et $T(K_v)/n \cong H^1(K_v,{_n}T)$ obtenues grâce à la suite de Kummer $0 \rightarrow {_n}T \rightarrow T \rightarrow T \rightarrow 0$ et à l'annulation de $H^1(K,T)$ et de $H^1(K_v,T)$ 
(qui viennent de ce que $K$ et $K_v$ sont de dimension cohomologique 
$1$), on a un isomorphisme $\text{Ker} \left( T(K)/n \rightarrow \prod_{v\in X^{(1)}} T(K_v)/n \right) \cong \Sha^1(K,{_n}T)$, d'où la suite exacte:
 $$0 \rightarrow \frac{{_n}\mathbb{P}^0(T)}{{_n}T(K)} \rightarrow {_n}A(\emptyset,T) \rightarrow \Sha^1(K,{_n}T) \rightarrow 0.$$
 En passant à la limite inductive sur $n$, on a la suite exacte:
 $$0 \rightarrow \varinjlim_n \frac{{_n}\mathbb{P}^0(T)}{{_n}T(K)} \rightarrow A(\emptyset,T)_{\rm tors} \rightarrow \varinjlim_n \Sha^1(K,{_n}T) \rightarrow 0.$$
 Il suffit alors d'identifier les groupes $\varinjlim_n \frac{{_n}\mathbb{P}^0(T)}{{_n}T(K)}$ et $\frac{\mathbb{P}^0(T)_{\text{tors}}}{T(K)_{\text{tors}}}$, ce qui vient à nouveau de l'exactitude de la limite inductive.
\end{proof}

\begin{lemma}\label{deuxacc}
On a des accouplements parfaits de groupes topologiques :
\begin{gather} \label{ev}
\frac{\mathbb{P}^0(T)_{{\rm tors}}}{T(K)_{{\rm tors}}} \times \frac{\varprojlim_n H^1(K,{_n}\tilde{T})}{\varprojlim_n \Sha^1(K,{_n}\tilde{T})} \rightarrow \mathbb{Q}/\mathbb{Z},\\
\label{sha}
\varinjlim_n \Sha^1(K,{_n}T) \times\varprojlim_n \Sha^1(K,{_n}\tilde{T}) \rightarrow\mathbb{Q}/\mathbb{Z}.
\end{gather}
\end{lemma}

\begin{proof}
Pour chaque $n\geq 1$, on écrit la suite de Poitou-Tate (\cite{diego1}, Th. 2.7 dans le cas $d=-1$):
$$0 \rightarrow \Sha^1(K,{_n}\tilde{T}) \rightarrow H^1(K,{_n}\tilde{T}) \rightarrow \bigoplus_{v\in X^{(1)}} H^1(K_v,{_n}\tilde{T}) \rightarrow {_n}T(K)^D \rightarrow 0.$$
On déduit un isomorphisme :
$$\frac{ H^1(K,{_n}\tilde{T})}{ \Sha^1(K,{_n}\tilde{T})} \cong \text{Ker}\left( \bigoplus_{v\in X^{(1)}} H^1(K_v,{_n}\tilde{T}) \rightarrow {_n}T(K)^D\right) .$$
Via l'identification $\bigoplus_{v\in X^{(1)}}H^1(K_v,{_n}\tilde{T}) \cong \left(\prod_{v\in X^{(1)}}H^0(K_v,{_n}T)\right) ^D \cong ({_n}\mathbb{P}^0(T))^D$, on obtient un isomorphisme :
$$\frac{ H^1(K,{_n}\tilde{T})}{ \Sha^1(K,{_n}\tilde{T})} \cong \left( \frac{{_n}\mathbb{P}^0(T)}{{_n}T(K)}\right) ^D.$$
En passant à la limite projective, on a:
$$\varprojlim_n \frac{ H^1(K,{_n}\tilde{T})}{ \Sha^1(K,{_n}\tilde{T})} \cong \left( \frac{\mathbb{P}^0(T)_{\text{tors}}}{T(K)_{\text{tors}}}\right) ^D.$$
Comme $\Sha^1(K,{_n}\tilde{T})$ est fini pour tout $n$, on a un isomorphisme:
$$\varprojlim_n \frac{ H^1(K,{_n}\tilde{T})}{ \Sha^1(K,{_n}\tilde{T})} \cong \frac{  \varprojlim_n H^1(K,{_n}\tilde{T})}{  \varprojlim_n \Sha^1(K,{_n}\tilde{T})}.$$
On obtient donc un accouplement parfait:
$$\frac{\mathbb{P}^0(T)_{\text{tors}}}{T(K)_{\text{tors}}} \times \frac{\varprojlim_n H^1(K,{_n}\tilde{T})}{\varprojlim_n \Sha^1(K,{_n}\tilde{T})} \rightarrow \mathbb{Q}/\mathbb{Z}.$$

Pour obtenir l'accouplement (\ref{sha}), il suffit de passer à la limite sur les accouplements parfaits de groupes finis:
$$\Sha^1(K,{_n}T) \times \Sha^1(K,{_n}\tilde{T}) \rightarrow\mathbb{Q}/\mathbb{Z}.$$
\end{proof}

\begin{theorem}\label{thtors}
L'accouplement (\ref{acc}) est un accouplement parfait entre une limite inductive de groupes profinis et une limite projective de groupes discrets de torsion.
\end{theorem}

\begin{proof}
On a un diagramme commutatif à colonnes exactes:\\
\centerline{\xymatrix{
0 \ar[d] & 0\ar[d] \\
\frac{\mathbb{P}^0(T)_{\text{tors}}}{T(K)_{\text{tors}}} \ar[d]\ar[r]^-{\cong}& \left( \frac{\varprojlim_n H^1(K,{_n}\tilde{T})}{\varprojlim_n \Sha^1(K,{_n}\tilde{T})}\right) ^D\ar[d] \\
 A(\emptyset,T)_{\text{tors}}\ar[d]\ar[r] &\left( \varprojlim_n H^1(K,{_n}\tilde{T})\right)^D\ar[d]\\
 \varinjlim_n \Sha^1(K,{_n}T)\ar[d]\ar[r]^-{\cong} & \left( \varprojlim_n \Sha^1(K,{_n}\tilde{T})\right) ^D\ar[d]\\
 0 & 0
}}
Le lemme des cinq permet donc de conclure.
\end{proof}

 \begin{corollary} \label{corcc}
 On a un accouplement naturel parfait entre une limite inductive de groupes profinis et une limite projective de groupes discrets de torsion:
 $$A(\emptyset,T)_{{\rm div, tors}} \times \varprojlim_n {_n}H^2(K,\hat{T}) \rightarrow \mathbb{Q}/\mathbb{Z}.$$
 \end{corollary}
 
 \begin{proof}
 On a une suite exacte:
 $$ 0 \rightarrow  \varprojlim_n H^0(K,\tilde{T})/n \rightarrow \varprojlim_n H^1(K,{_n}\tilde{T}) \rightarrow \varprojlim_n {_n}H^1(K,\tilde{T}) \rightarrow 0.$$
 Or $H^0(K,\tilde{T})/n$ s'identifie à $H^1(K,\hat{T})/n$ et $H^1(K,\hat{T})$ est d'exposant fini. On obtient donc une suite exacte:
  $$ 0 \rightarrow  H^1(K,\hat{T}) \rightarrow \varprojlim_n H^1(K,{_n}\tilde{T}) \rightarrow \varprojlim_n {_n}H^1(K,\tilde{T}) \rightarrow 0,$$
  où $\varprojlim_n {_n}H^1(K,\tilde{T})$ est sans torsion. On déduit alors le résultat du théorème \ref{thtors} et de l'isomorphisme $H^1(K,\tilde{T}) \cong H^2(K,\hat{T})$.
  
 \end{proof}
 
 \begin{remarque}
 Dans le cas $K=\mathbb{G}_m$, on obtient un accouplement parfait:
  $$A(\emptyset,\mathbb{G}_m)_{\text{div, tors}} \times \text{Hom}_c(\text{Gal}(\overline{K}/K),\hat{\mathbb{Z}}) \rightarrow \mathbb{Q}/\mathbb{Z}.$$
 \end{remarque}
 
 \begin{conclusion}
 Le groupe $A(\emptyset, T)$ est somme directe d'un groupe uniquement divisible $D$ et d'un groupe $B(T)$ tel que:
 \begin{gather*}
 B(T)_{\text{div}} \cong \left( \varprojlim_n {_n}H^2(K,\hat{T})\right)^D,\\
 \overline{B(T)}_{\rm tors} \cong H^1(K,\hat{T}) ^D,\;\;\;
 \overline{B(T)}/\overline{B(T)}_{\rm tors} \cong \mathbb{Z}^{\text{rg}(H^0(k,\hat{T}))}.
 \end{gather*}
 \end{conclusion}

\section{Cas $k=\mathbb{C}((t))$}
    
Dans cette section, on suppose que $K$ est le corps des fonctions d'une 
courbe projective lisse $X$ sur le corps $k=\mathbb{C}((t))$. La th\'eorie 
du corps de classes pour $K$ (cf. \cite{adt}, section I, appendice A) 
pr\'esente des similitudes avec celle 
d'un corps de fonctions sur un corps fini, mais \'egalement des 
diff\'erences (en particulier le th\'eor\`eme de Brauer-Hasse-Noether et 
le th\'eor\`eme de \v Cebotarev ne valent en g\'en\'eral pas). Ceci va 
se retrouver dans la description de l'espace ad\'elique d'un $K$-tore.
    
\begin{proposition}\label{c((t))}
Soit $C_K:=A(\emptyset,\G)=\left( \prod_{v\in C^{(1)}}' 
K_v^* \right)/K^*$. 
Alors il existe un sous-groupe divisible $U^0$ de $C_K$ tel que 
$(C_K/U^0)_{\rm tors}$ soit de torsion de type cofini.
\end{proposition} 

\begin{proof}
Pour chaque $v\in X^{(1)}$, soit $ U_v^0$ le sous-groupe de $\mathcal{O}_v^*$ constitué des éléments dont la réduction dans le corps résiduel $k(v)$ est de valuation nulle. Soit $U^0$ l'image de $\prod_{v\in X^{(1)}} U_v^0$ dans $C_K$. Posons $C_K^0=C_K/U^0$. Le groupe $U^0$ est divisible. Montrons que $(C_K/U^0)_{\text{tors}}$ est de torsion de type cofini.\\

Soit $k(v)$ le corps r\'esiduel de $v$ et ${\rm val}_{k(v)}$ sa valuation
(ici $k(v)$ est vu comme un corps complet pour une valuation discr\`ete,
via le fait que c'est une extension finie de $\CC((t))$).

Le groupe $C_K^0$ s'insère dans une suite exacte, qu'on obtient comme la 
suite (\ref{gmsuite})~:
\begin{equation}\label{c^0}
0 \rightarrow \frac{\prod_{v\in X^{(1)}} \mathbb{Z}}{\mathbb{Z}.(e_v)}
\rightarrow C_K^0 \rightarrow \text{Pic}\, X \rightarrow 0,
\end{equation}
o\`u l'\'el\'ement $(e_v) \in \bigoplus_{v \in X^{(1)}} \Z \subset 
\prod_{v \in X^{(1)}} \Z$ est obtenu 
en posant $e_v={\rm val}_{k(v)}(t)$.
 
En passant aux sous-groupes de torsion, on obtient une suite exacte:
\begin{equation}\label{c^00}
0 \to F \rightarrow (C_K^0)_{\rm tors} \rightarrow (\text{Pic}\, X)_{\rm tors}.
\end{equation}
dans laquelle le groupe $F$ est fini. 
De plus, $ (\pic X)_{\rm tors}=J(k)_{\rm tors}$, 
o\`u $J$ est la jacobienne de $X$. Mais $J(\kbar)_{\rm tors}$ est d\'ej\`a 
de type cofini car pour $n >0$, le sous-groupe de $n$-torsion d'une 
vari\'et\'e ab\'elienne de dimension $g$ sur un corps alg\'ebriquement 
clos de caract\'eristique z\'ero est isomorphe \`a $(\Z/n\Z)^{2g}$. 
Ainsi, $(C_K^0)_{\rm tors}$ est aussi de type cofini.

\end{proof}

\begin{proposition} \label{div2}
Soit $T$ un $K$-tore. Alors $A(\emptyset,T)_{\infty \text{-div}} = A(\emptyset,T)_{ \text{div}}$.
\end{proposition}

\begin{proof}
D'après le lemme d'Ono, il existe un entier $m>0$ et une suite exacte:
    $$0\rightarrow F \rightarrow R_0 \rightarrow T^m  \times R_1 \rightarrow 0$$
    tels que $F$ est un groupe abélien fini étale et $R_0$ et $R_1$ sont des tores quasi-triviaux. On obtient alors un diagramme commutatif à lignes exactes:\\
    \small
  \centerline{  \xymatrix{
  0 \ar[r] & F(K) \ar[r] \ar[d] & R_0(K) \ar[r]\ar[d] & T(K)^m \times R_1(K) \ar[r]\ar[d] & H^1(K,F) \ar[r] \ar[d]& 0\\
  0 \ar[r] & \prod_{v\in X^{(1)}} F(K_v) \ar[r] & \prod_{v\in X^{(1)}} ' R_0(K_v) \ar[r] & \prod_{v\in X^{(1)}} ' (T(K_v)^m \times R_1(K_v)) \ar[r] & \prod'_{v\in X^{(1)}} H^1(K_v,F) \ar[r] & 0
  }}
  \normalsize
  Une chasse au diagramme permet de montrer qu'on a alors une suite exacte:
$$  0 \rightarrow N \rightarrow A(\emptyset,R_0) \rightarrow A(\emptyset,T)^m\times A(\emptyset,R_1) \rightarrow \left( \prod'_{v\in X^{(1)}} H^1(K_v,F) \right) /H^1(K,F) \rightarrow 0.$$
  avec $N$ un groupe abélien d'exposant fini. D'après le lemme \ref{infdiv0}, il suffit de démontrer que $(A(\emptyset,R_0)/N)_{\infty \text{-div}} = (A(\emptyset,R_0)/N)_{ \text{div}}$.\\
  
   D'après la proposition \ref{c((t))}, il existe un sous-groupe divisible $D$ de $A(\emptyset,R_0)$ tel que $(A(\emptyset,R_0)/D)_{\text{tors}}$ est de torsion de type cofini. Comme $N$ est de torsion, on déduit que $((A(\emptyset,R_0)/N)/(D/(D\cap N)))_{\text{tors}}$ est aussi de torsion de type cofini. Comme $D/(D\cap N)$ est divisible, en appliquant le lemme \ref{tcof}, on en déduit que $(A(\emptyset,R_0)/N)_{\infty \text{-div}} = (A(\emptyset,R_0)/N)_{ \text{div}}$, ce qui achève la preuve.
\end{proof}

Soit $T$ un $K$-tore. 
Rappelons que pour tout $v \in X^{(1)}$, on a une dualit\'e locale 
entre $H^0(K_v,T)_{\wedge}$ et $H^2(K_v,\hat T)$ (\cite{dhct}, Prop. 3.4), 
d'o\`u l\`a encore une fl\`eche de r\'eciprocit\'e (induite par 
un accouplement d\'efini de mani\`ere similaire \`a (\ref{reciproc}))
$r : \mathbb{P}^0(T) \to H^2(K,\hat T)^D$. Le th\'eor\`eme suivant 
dit que le noyau de $r$ est divisible modulo $T(K)$ et d\'ecrit l'adh\'erence de 
son image. 

\begin{theorem} \label{cttheo}
La fl\`eche $r$ induit un morphisme 
injectif $\overline{A(\emptyset,T)} \to H^2(K,\hat T)^D$, dont 
l'adh\'erence de l'image $I_{\rm adh}$ s'ins\`ere dans une suite exacte 
$$0 \to I_{\rm adh}  \to H^2(K,\hat T)^D \to \Sha^2(K,\hat T)^D \to 0.$$
\end{theorem}

\begin{proof}
D'apr\`es \cite{diego1}, Th.~3.20, 
on a une suite exacte 
\begin{equation}\label{ptc((t))}
H^0(K,T)_{\wedge} \to \mathbb{P}^0(T)_{\wedge} \to C \to 0, 
\end{equation}
o\`u $C$ est le groupe profini d\'efini par 
$C:=\ker [H^2(K,\hat T)^D \to \Sha^2(K,\hat T)^D]$. On remarquera que l'énoncé de \cite{diego1}, Th.~3.20 exige comme hypothèse l'existence d'une extension finie $L$ de $K$ déployant $T$ telle que $\Sha^2(L,\mathbb{G}_m)=0$. Mais sa preuve montre que cette hypothèse n'est pas nécessaire pour avoir l'exactitude de (\ref{ptc((t))}). Le r\'esultat 
d\'ecoule alors du lemme~\ref{completion} joint \`a la proposition~\ref{div2}.
\end{proof}

Concernant le sous-groupe de torsion de $A(\emptyset,T)$, on peut procéder 
de manière similaire à la section \ref{ctors}. On obtient 
le théorème suivant:

\begin{theorem}\label{c((t))tors}
Il existe un accouplement parfait entre une limite inductive de groupes profinis et une limite projective de groupes discrets de torsion:
\begin{equation}\label{acc2}
 A(\emptyset,T)_{\rm tors}  \times \frac{\varprojlim_n H^2(K,\hat{T}/n)}{{\rm Im}\left( \overline{\Sha^2(K,\hat{T})}\right) } \rightarrow \mathbb{Q}/\mathbb{Z}.
 \end{equation}
\end{theorem}

Dans cet énoncé, le groupe $\overline{\Sha^2(K,\hat{T})}$ est fini (cf. \cite{dhct}, Th. 7.2).

\begin{proof}
La preuve suit les mêmes lignes que celle du théorème \ref{thtors}. On commence par construire un accouplement naturel:
$$ A(\emptyset,T)_{\rm tors}  \times \varprojlim_n H^2(K,\hat{T}/n) \rightarrow \mathbb{Q}/\mathbb{Z}$$
tout à fait analogue à l'accouplement (\ref{acc}), à condition de remplacer
la dualit\'e de Poincar\'e entre $H^1_c(U,_n {\mathcal T})$ et
$H^1(U,_n \tilde {\mathcal T})$ par la dualit\'e d'Artin-Verdier 
(\cite{dhct}, Proposition~5.1)
entre $H^1_c(U,_n {\mathcal T})$ et $H^2(U, \hat {\mathcal T}/n)$. Il faut ensuite remarquer que, dans la preuve du lemme \ref{ators}, le groupe $\text{Ker} \left( T(K)/n \rightarrow \prod_{v\in X^{(1)}} T(K_v)/n \right)$ ne s'identifie plus à $\Sha^1(K,{_n}T)$, mais à $\text{Ker} \left( \Sha^1(K,{_n}T) \rightarrow {_n}\Sha^1(K,T)\right) $. Ainsi, la suite exacte du lemme \ref{ators} doit être remplacée par une suite exacte:
$$0 \rightarrow \frac{\mathbb{P}^0(T)_{{\rm tors}}}{T(K)_{{\rm tors}}}\rightarrow  A(\emptyset,T)_{\rm tors}\rightarrow \text{Ker} \left(  \varinjlim_n \Sha^1(K,{_n}T) \rightarrow \Sha^1(K,T) \right) \rightarrow 0.$$
Par conséquent, les accouplements du lemme \ref{deuxacc} doivent être remplacés par des accouplements parfaits:
\begin{gather*} 
\frac{\mathbb{P}^0(T)_{{\rm tors}}}{T(K)_{{\rm tors}}} \times \frac{\varprojlim_n H^2(K,\hat{T}/n)}{\varprojlim_n \Sha^2(K,\hat{T}/n)} \rightarrow \mathbb{Q}/\mathbb{Z},\\
\text{Ker} \left(  \varinjlim_n \Sha^1(K,{_n}T) \rightarrow \Sha^1(K,T) \right) \times \frac{\varprojlim_n \Sha^2(K,\hat{T}/n)}{{\rm Im}\left( \overline{\Sha^2(K,\hat{T})}\right) } \rightarrow\mathbb{Q}/\mathbb{Z}.
\end{gather*}
On termine la preuve en utilisant le lemme des cinq comme dans la démonstration du théorème \ref{thtors}.
\end{proof}

\begin{corollary} \label{corcct}
On a un accouplement parfait entre une limite inductive de groupes profinis et une limite projective de groupes discrets de torsion:
$$
 A(\emptyset,T)_{{\rm div, tors}}  \times \varprojlim_n {_n}H^3(K,\hat{T}) \rightarrow \mathbb{Q}/\mathbb{Z},
$$
et le groupe $\overline{A(\emptyset,T)_{\rm tors}}$ est d'exposant fini.
\end{corollary}

\begin{proof}
On a une suite exacte:
$$0 \rightarrow \varprojlim_n H^2(K,\hat{T})/n \rightarrow \varprojlim_n H^2(K,\hat{T}/n) \rightarrow \varprojlim_n {_n}H^3(K,\hat{T}) \rightarrow 0$$
(le passage \`a la limite projective ne pose pas de probl\`emes car les 
fl\`eches de transition dans le terme de gauche sont surjectives).
Si $L$ est une extension finie déployant $T$, le groupe de Galois $\text{Gal}(\overline{L}/L)$ n'a pas de torsion d'apr\`es le th\'eor\`eme d'Artin-Schreier 
car $L$ contient $\sqrt{-1}$. Par conséquent, le groupe $H^2(L,\mathbb{Z})$ est divisible. Un argument de restriction-corestriction implique donc que $\varprojlim_n H^2(K,\hat{T})/n$ est d'exposant fini. De plus, le groupe $\varprojlim_n {_n}H^3(K,\hat{T})$ est sans torsion. Le corollaire découle donc immédiatement du théorème \ref{c((t))tors} et du fait que 
l'image de $H^2(K,\hat T)$
(donc a fortiori aussi celle de $\Sha^2(K,\hat{T})$)
dans $\varprojlim_n H^2(K,\hat{T}/n)$ est contenue dans $\varprojlim_n H^2(K,\hat{T})/n$.
\end{proof}

    \section{Cas $k=\mathbb{Q}_p$}
     
    Supposons que $k$ soit un corps $p$-adique. On va d\'emontrer 
l'existence d'une suite exacte analogue \`a celles des th\'eor\`emes 
\ref{kalgclos} et \ref{cttheo}.
La principale difficult\'e est ici de montrer que le sous-groupe 
$A(\emptyset,T)_{\div}$ co\"{\i}ncide avec $A(\emptyset,T)_{\infty-\div}$,
ce qui s'av\`ere plus ardu que dans les cas trait\'es 
pr\'ec\'edemment. On commence comme d'habitude par le cas 
$T=\G$.

\begin{lemma}\label{eldiv}
Soit $D_K=(C_K)_{\infty-\div}$ le sous-groupe des éléments divisibles de 
$C_K:=A(\emptyset,\G)=\left( \prod_{v\in X^{(1)}}' 
K_v^* \right)/K^*$. 
On a alors une suite exacte:
$$0 \rightarrow D_K \rightarrow C_K \rightarrow (\text{Br}\; K)^D.$$
\end{lemma}

\begin{proof}
Cela découle immédiatement de la suite exacte (\cite{dhsza2}, Th. 2.9. dans 
le cas $T=\G$)

$$0 \rightarrow (K^*)_{\wedge} \rightarrow \left( \prod_{v\in X^{(1)}}' K_v^* \right)_{\wedge} \rightarrow (\text{Br}\; K)^D$$
et du lemme~\ref{completion}, b).
\end{proof}

\begin{theorem}\label{qp}
Pour $v\in X^{(1)}$, soient $\mathfrak{m}_v$ l'idéal maximal de $\mathcal{O}_v$ et $U^1_v$ le groupe $1+\mathfrak{m}_v$. On a alors une suite exacte:
$$0 \rightarrow \prod_{v\in X^{(1)}} U^1_v \rightarrow C_K \rightarrow (\text{Br}\, K)^D,$$
avec $D_K=\prod_{v\in X^{(1)}} U^1_v$ uniquement divisible; de plus 
$D_K$ est le sous-groupe divisible maximal de $C_K$.
\end{theorem}

\begin{proof}
Nous avons des suites exactes:
\begin{gather*}
0 \rightarrow \frac{\prod_{v\in X^{(1)}}\mathcal{O}_v^*}{k^*} \rightarrow C_K \rightarrow \text{Pic}\, X \rightarrow 0,\\
0 \rightarrow C' \rightarrow \text{Pic} \, X \rightarrow E \rightarrow 0,
\end{gather*}
où $C' \cong \mathbb{Z}_p^m \oplus \Z$  
pour un certain $m \geq 0$ et $E$ est fini~: en effet, si $J$ est la 
jacobienne de $X$, on a $\pic^0 X=J(k)$, qui est de la forme 
$\mathbb{Z}_p^m \oplus E$ avec $E$ fini d'apr\`es \cite{mattuck}.

\smallskip

Soit $k(v)$ le corps r\'esiduel de $v \in X^{(1)}$. 
D'après le lemme \ref{infdiv}, le groupe  $D_K=(C_K)_{\infty-\div}$
coïncide avec le sous-groupe $\mathcal{D}_K$ des éléments divisibles de $\frac{\prod_{v\in X^{(1)}}\mathcal{O}_v^*}{k^*}$. En exploitant l'isomorphisme $\mathcal{O}_v^* \cong U^1_v \times k(v)^*$, on obtient un isomorphisme:
$$\frac{\prod_{v\in X^{(1)}}\mathcal{O}_v^*}{k^*}\cong \prod_{v\in X^{(1)}} U^1_v \times \frac{\prod_{v\in X^{(1)}}k(v)^*}{k^*}.$$
Or le groupe $\prod_{v\in X^{(1)}} U^1_v$ est uniquement divisible 
via \cite{col}, chapitre IV, Prop. 6 (jointe au fait que 
le corps r\'esiduel $k(v)$ de $K_v$ est de 
caract\'eristique z\'ero). 
En utilisant le lemme~\ref{eldiv}, il nous suffit donc pour conclure de
montrer la proposition suivante~:

\begin{proposition}
Posons $B=\frac{\prod_{v\in X^{(1)}}k(v)^*}{k^*}$. Alors 
$B_{\infty-\div}=0$.
\end{proposition}

Notons $\mathcal{R}$ (resp. $\mathcal{R}_v$) l'anneau des entiers 
de $k$ (resp. $k(v)$) et $\mathfrak{n}$ (resp. $\mathfrak{n}_v$) 
l'idéal maximal de $\mathcal{R}$ (resp. $\mathcal{R}_v$). 
On  note aussi $V^1$ (resp. $V^1_v$) le groupe 
$1 + \mathfrak{n}$ (resp. $1+\mathfrak{n}_v$), puis 
$\mathbb{F}$ (resp. $\mathbb{F}_v$) le corps résiduel de $k$ (resp. $k(v)$). 
En choisissant $v_0\in X^{(1)}$ et en notant $e_{v_0}$ l'indice de ramification de l'extension $k(v)/k$, on a des suites exactes:
\begin{gather*}
0 \rightarrow \frac{\prod_{v\in X^{(1)}}\mathcal{R}_v^*}{\mathcal{R}^*} \rightarrow \frac{\prod_{v\in X^{(1)}}k(v)^*}{k^*} \rightarrow \frac{\prod_{v\in X^{(1)}} \mathbb{Z}}{\mathbb{Z}} \rightarrow 0,\\
0 \rightarrow \prod_{v\in X^{(1)}\setminus \{v_0\}} \mathbb{Z} \rightarrow \frac{\prod_{v\in X^{(1)}} \mathbb{Z}}{\mathbb{Z}} \rightarrow \frac{\mathbb{Z}}{e_{v_0}\mathbb{Z}} \rightarrow 0.
\end{gather*}
En appliquant encore le lemme \ref{infdiv}, on voit que le sous-groupe des éléments divisibles de $B=\frac{\prod_{v\in X^{(1)}}k(v)^*}{k^*}$ coïncide avec celui de $\frac{\prod_{v\in X^{(1)}}\mathcal{R}_v^*}{\mathcal{R}^*}$. On a de plus une suite exacte:
\begin{equation} \label{scindsuite}
0 \rightarrow  \frac{\prod_{v\in X^{(1)}}V^1_v}{V_1} \rightarrow \frac{\prod_{v\in X^{(1)}}\mathcal{R}_v^*}{\mathcal{R}^*} \rightarrow \frac{\prod_{v\in X^{(1)}}\mathbb{F}_v^*}{\mathbb{F}^*}\rightarrow 0,
\end{equation}
qui est scindée par le relèvement de Teichmüller. En écrivant les suites exactes:
\begin{gather*}
0 \rightarrow  \prod_{v\in X^{(1)}\setminus \{v_0\}}V^1_v\rightarrow  \frac{\prod_{v\in X^{(1)}}V^1_v}{V_1} \rightarrow \frac{V^1_{v_0}}{V^1} \rightarrow 0, \\
0 \rightarrow \prod_{v\in X^{(1)}\setminus\{v_0\}}\mathbb{F}_v^* \rightarrow \frac{\prod_{v\in X^{(1)}}\mathbb{F}_v^*}{\mathbb{F}^*} \rightarrow \frac{\mathbb{F}_{v_0}^*}{\mathbb{F}^*} \rightarrow 0,
\end{gather*}
et en remarquant que $\prod_{v\in X^{(1)}\setminus \{v_0\}}V^1_v$ et $\prod_{v\in X^{(1)}\setminus\{v_0\}}\mathbb{F}_v^*$ n'ont pas d'éléments infiniments divisibles non triviaux et que $\frac{V^1_{v_0}}{V^1}$ est un $\mathbb{Z}_p$-module de type fini, le lemme \ref{infdiv0} montre que $\frac{\prod_{v\in X^{(1)}}V^1_v}{V_1}$ et $\frac{\prod_{v\in X^{(1)}}\mathbb{F}_v^*}{\mathbb{F}^*}$ n'ont pas d'éléments divisibles non triviaux. Comme la suite
exacte (\ref{scindsuite}) 
est scind\'ee, on en d\'eduit qu'il en va de m\^eme de 
$\frac{\prod_{v\in X^{(1)}}\mathcal{R}_v^*}{\mathcal{R}^*}$, 
et donc que $B_{\infty-\div}=0$ comme on voulait.

\end{proof}

\begin{corollary} \label{qtdiv}
Si $R$ est un $K$-tore quasi-trivial, alors $A(\emptyset,R)_{\infty-\div}=
A(\emptyset,R)_{\div}$, et ce groupe est uniquement divisible.
\end{corollary}

Soit maintenant $T$ un $K$-tore quelconque. 
D'après le lemme d'Ono, il existe un entier $m>0$ et une suite exacte:
    $$0\rightarrow F \rightarrow R_0 \rightarrow T^m  \times R_1 \rightarrow 0$$
    tels que $F$ est un $K$-groupe abélien fini, tandis que 
$R_0$ et $R_1$ sont des tores quasi-triviaux. On fixe une telle suite jusqu'\`a la fin de cette section.

\begin{proposition}
Le sous-groupe divisible maximal de $A(\emptyset,T)$ est uniquement divisible.
\end{proposition}

\begin{proof}
    On consid\`ere le diagramme commutatif à lignes exactes,
et dont les trois premi\`eres fl\`eches verticales sont injectives~: \\

\begin{equation} \label{diagqp}
  \centerline{  \xymatrix{
  0 \ar[r] & F(K) \ar[r] \ar[d] & R_0(K) \ar[r]\ar[d] & T(K)^m \times R_1(K) \ar[r]\ar[d] & H^1(K,F) \ar[r] \ar[d]& 0\\
  0 \ar[r] & \prod_{v\in X^{(1)}} F(K_v) \ar[r] & \prod_{v\in X^{(1)}} ' R_0(K_v) \ar[r] & \prod_{v\in X^{(1)}} ' (T(K_v)^m \times R_1(K_v)) \ar[r] & \prod'_{v\in X^{(1)}} H^1(K_v,F) \ar[r] & 0
  }}
\end{equation}

\smallskip

  Une chasse au diagramme permet de montrer qu'on a alors une suite exacte:
\begin{equation} \label{exactm}
  0 \rightarrow N \rightarrow A(\emptyset,R_0) \rightarrow A(\emptyset,T)^m\times A(\emptyset,R_1) \rightarrow \left( \prod'_{v\in X^{(1)}} H^1(K_v,F) \right) /H^1(K,F) \rightarrow 0.
\end{equation}
  avec $N$ un groupe abélien d'exposant fini (tout comme le groupe de 
droite). La proposition découle 
alors du lemme \ref{ud} et du corollaire~\ref{qtdiv}.
\end{proof}

\begin{theorem} \label{qptheo}
Soit $T$ un $K$-tore. Alors 
$A(\emptyset,T)_{\infty-\div}=A(\emptyset,T)_{\div}$. 
Ce groupe est uniquement divisible.
\end{theorem}

\begin{proof} Elle va consister en plusieurs r\'eductions 
successives. Pour simplifier les notations, convenons qu'un groupe ab\'elien $A$
{\it v\'erifie la propri\'et\'e (DD)} si $A_{\infty-\div}=A_{\div}$.
Soit $M$ l'image de $\prod_{v\in X^{(1)}} F(K_v)$ 
dans $A(\emptyset,R_0)$.

\smallskip

{\bf \'Etape 1}~: On va montrer qu'il suffit d'avoir la propri\'et\'e 
(DD) pour le groupe $A(\emptyset,R_0)/M$.

\smallskip

On utilise la suite exacte (\ref{exactm}), o\`u
  le groupe de droite est d'exposant fini. D'après le lemme \ref{infdiv0} et 
le corollaire~\ref{qtdiv}, il suffit de démontrer que $A(\emptyset,R_0)/N$ a la propri\'et\'e (DD).\\

  Remarquons maintenant que via le diagramme (\ref{diagqp}), le groupe 
$N$ s'insère dans une suite exacte:
  $$0\rightarrow M \rightarrow N \rightarrow \Sha^1(K,F) \rightarrow 0.$$
  On a alors une suite exacte:
  $$0 \rightarrow \Sha^1(K,F) \rightarrow A(\emptyset,R_0)/M \rightarrow A(\emptyset,R_0)/N \rightarrow 0.$$
  D'après le lemme \ref{fini}, comme le groupe $\Sha^1(K,F)$ est fini, il suffit bien de démontrer que $A(\emptyset,R_0)/M$ a la propri\'et\'e (DD).\\

\smallskip
  
   Écrivons $R_0 = \prod_{i=1}^m R_{L_i/K}(\mathbb{G}_m)$. Pour chaque $i$, 
on se donne une courbe projective lisse géométriquement intègre $X_i$ sur une 
extension finie $k_i$ de $k$, 
telle que $L_i=k_i(X_i)$. L'extension $L_i/K$ induit un morphisme $\pi_i: X_i \rightarrow X$. Pour tout point ferm\'e $w$ de 
$X_i$, on note $\calo_{i,w}$ le compl\'et\'e 
de l'anneau local de $X_i$ en $w$ et $k_i(w)$ son corps r\'esiduel. On note aussi $\mathfrak{m}_{i,w}$ l'idéal maximal de $\calo_{i,w}$ et $U^1_{i,w}$ le groupe $1+\mathfrak{m}_{i,w}$. On note ensuite $\mathcal{R}_i$ (resp. $\mathcal{R}_{i,w}$) l'anneau des entiers de $k_i$ (resp. $k_i(w)$), et $\mathfrak{n}_i$ (resp. $\mathfrak{n}_{i,w}$) l'idéal maximal de $\mathcal{R}_i$ (resp. $\mathcal{R}_{i,w}$). Soient finalement $V^1_i$ (resp. $V^1_{i,w}$) le groupe $1 + \mathfrak{n}_i$ (resp. $1+\mathfrak{n}_{i,w}$), et $\mathbb{F}_i$ (resp. $\mathbb{F}_{i,w}$) le corps résiduel de $k_i$ (resp. $k_i(w)$). 

\smallskip

{\bf \'Etape 2~:} Pour chaque $v\in X^{(1)}$, notons $\tilde{M}_v$ l'image de $F(K_v)$ dans $R_0(K_v)$. Pour chaque $v\in X^{(1)}$, on a un 
isomorphisme (non canonique):
$$R_0(K_v) \cong \prod_{i=1}^m \prod_{w\in X_i^{(1)} \cap \pi_i^{-1}(v)} L_{i,w}^* \cong \prod_{i=1}^m \prod_{w\in X_i^{(1)} \cap \pi_i^{-1}(v)} (\mathbb{Z} \times U^1_{i,w} \times \mathbb{Z} \times \mathcal{R}_{i,w}^*).$$
En particulier, le conoyau de l'injection $j_v:\prod_{i=1}^m \prod_{w\in X_i^{(1)} \cap \pi_i^{-1}(v)} \mathcal{R}_{i,w}^*\hookrightarrow R_0(K_v)$ est sans 
torsion. Le groupe $F(K_v)$ étant d'exposant fini, $\tilde{M}_v$ est contenu dans l'image de $j_v$. Des dévissages parfaitement analogues à ceux du théorème \ref{qp} montrent alors qu'il suffit de prouver que le groupe:
$$ Z:=\frac{ \prod_{v\in X^{(1)}} \frac{\prod_{i=1}^m \prod_{w\in X_i^{(1)} \cap \pi_i^{-1}(v)} \mathcal{R}_{i,w}^*}{\tilde{M}_v} }{\text{Im}\left( \prod_{i=1}^m \mathcal{R}_i^*\right) }  $$
a la propriété (DD).
\smallskip

{\bf \'Etape 3~:} On a une suite exacte:
\begin{gather*} 
0 \rightarrow T \rightarrow Z \rightarrow V \rightarrow 0,\\
\text{avec} \;\;\;\; T:= \frac{ \prod_{v\in X^{(1)}} \frac{\prod_{i=1}^m \prod_{w\in X_i^{(1)} \cap \pi_i^{-1}(v)} (\mathcal{R}_{i,w}^*)_{\rm tors}}{\tilde{M}_v} }{\text{Im}\left( \prod_{i=1}^m (\mathcal{R}_i^*)_{\rm tors}\right) },\;\;\;\; 
V:= \prod_{i=1}^m \frac{\prod_{w\in X_i^{(1)}} (V^1_{i,w}/V^1_{i,w, \rm tors})}{ (V^1_{i}/V^1_{i, \rm tors})}. 
\end{gather*} Comme chaque $V^1_{i,w}$ et chaque $V^1_{i}$ est un $\mathbb{Z}_p$-module de type fini, le lemme \ref{infdiv} montre qu'il suffit de vérifier que le groupe $T$ vérifie la propriété (DD). Or le groupe $\prod_{i=1}^m (\mathcal{R}_i^*)_{\rm tors}$ est fini et le groupe $\prod_{v\in X^{(1)}} \frac{\prod_{i=1}^m \prod_{w\in X_i^{(1)} \cap \pi_i^{-1}(v)} (\mathcal{R}_{i,w}^*)_{\rm tors}}{\tilde{M}_v} $ vérifie la propriété (DD) car c'est un produit de groupes finis. On conclut en appliquant le lemme \ref{fini}.

\end{proof}

Soit $\mathbb{P}^0(T) = \prod_{v\in X^{(1)}}' T(K_v)$ l'espace ad\'elique 
d'un $K$-tore $T$.  
La dualit\'e locale entre
$H^0(K_v,T)$ et $H^2(K_v,T')$ (\cite{dhsza1}, Prop~2.2) induit un accouplement
${\mathbb P}^0(T) \times H^2(K,T') \to \Q/\Z,$
d'o\`u une fl\`eche de r\'eciprocit\'e $r : {\mathbb P}^0(T) \to
H^2(K,T')^D$. Cette fl\`eche a encore un noyau divisible modulo $T(K)$, et ici 
l'adh\'erence de son image est d'indice fini dans $H^2(K,T')^D$~:

\begin{corollary} \label{qpcor}
Soit $T$ un $K$-tore. L'application $r$ induit une fl\`eche injective 
$\overline{A(\emptyset,T)} \to H^2(K,T')^D$. De plus, l'adh\'erence 
$I_{\rm adh}$ de l'image $I$ de $r$ s'ins\`ere dans une suite exacte
$$0 \to I_{\rm adh} \to H^2(K,T')^D \to \Sha^1(T) \to 0.$$
\end{corollary}

\begin{proof}
D'apr\`es \cite{dhsza2}, Th.~2.9, on a une suite exacte
$$0 \to H^0(K,T)_{\wedge} \to \mathbb{P}^0(T)_{\wedge} \to H^2(K,T')^D 
\to \Sha^1(T) \to 0,$$
avec $\Sha^1(T)$ fini et dual de $\Sha^2(T')$. Il suffit alors 
d'appliquer le lemme~\ref{completion} et le th\'eor\`eme~\ref{qptheo}.
\end{proof}

Concernant le sous-groupe de torsion de $A(\emptyset,T)$, on peut encore 
procéder comme dans la section \ref{ctors}, en utilisant cette fois-ci 
la dualit\'e d'Artin-Verdier entre $H^1_c(U,_n {\mathcal T})$ et 
$H^3(U,_n {\mathcal T}')$ (cf. \cite{dhsza1}, preuve du Th. 1.3). 
D'o\`u le théorème suivant:

\begin{theorem}\label{Qptors}
Il existe un accouplement parfait entre une limite inductive de groupes profinis et une limite projective de groupes discrets de torsion:
\begin{equation}\label{acc3}
 A(\emptyset,T)_{\rm tors}  \times \frac{\varprojlim_n H^3(K,{_n}T')}{{\rm Im}(\Sha^2(K,T'))} \rightarrow \mathbb{Q}/\mathbb{Z}.
 \end{equation}
\end{theorem}

Dans cet énoncé, le groupe $\Sha^2(K,T')$ est fini (cf. \cite{dhsza1}, Prop. 3.4).

\begin{corollary} \label{corccqp}
Le groupe $H^3(K,T')$ est nul et on a un accouplement parfait entre une limite inductive de groupes profinis et une limite projective de groupes discrets de torsion:
$$
 A(\emptyset,T)_{\rm tors}  \times \frac{H^2(K,T')_{\wedge}}{{\rm Im}(\Sha^2(K,T'))} \rightarrow \mathbb{Q}/\mathbb{Z}.$$
\end{corollary}

\begin{proof}
On a une suite exacte:
$$0 \rightarrow \varprojlim_n H^2(K,T')/n \rightarrow \varprojlim_n H^3(K,{_n}T') \rightarrow \varprojlim_n {_n}H^3(K,T') \rightarrow 0.$$
Le groupe $\varprojlim_n {_n}H^3(K,T')$ n'a pas de torsion et 
l'image de 
$H^2(K,T')$ (donc aussi celle de $\Sha^2(K,T')$)
dans $\varprojlim_n H^3(K,{_n}T')$ est contenue dans $\varprojlim_n H^2(K,T')/n$. De plus, d'après le théorème \ref{qptheo}, le groupe $A(\emptyset,T)_{\text{tors},\text{div}}$ est nul. Donc $\varprojlim_n {_n}H^3(K,T')=0$. En outre, par dimension cohomologique, le groupe $H^3(K,T')$ est divisible car $H^4(K,_n T')=0$. On en déduit que $H^3(K,T')=0$, puis que $ H^2(K,T')_{\wedge} \cong \varprojlim_n H^3(K,{_n}T')$, ce qui achève la preuve.
\end{proof}

\begin{remarque}
Un résultat d'U. Jannsen (\cite{Jan}, points 2) et 3) de la page 128) affirme que $H^3(K,\mathbb{G}_m)=0$. Cela implique la nullité de $H^3(K,T)$ pour tout $K$-tore $T$ via un argument de restriction-corestriction. Le corollaire précédent 
permet de retrouver ce résultat. On peut aussi l'obtenir plus facilement 
via la suite exacte de Faddeev (cf. \cite{dukedh}, p. 241) 
quand $X={\mathbb P^1}$, 
argument qui s'\'etend sans trop de difficult\'es \`a une courbe quelconque 
en exploitant la nullit\'e de $H^2(k,\pic^0 X)$ (laquelle d\'ecoule 
de \cite{adt}, Cor. I.3.4). 
\end{remarque}

\begin{remarque}
Dans le cas $T=\mathbb{G}_m$, on obtient un accouplement parfait :
$$
 A(\emptyset,\mathbb{G}_m)_{\text{tors}}  \times (\text{Br}\, K)_{\wedge} \rightarrow \mathbb{Q}/\mathbb{Z}.
$$
\end{remarque}

\textit{Remerciements.} Nous tenons à remercier B. Kahn pour 
nous avoir expliqué que le faisceau $\mathbb{Q}/\mathbb{Z}(-1)$ 
est le bon analogue du complexe motivique $\mathbb{Z}(d)$ 
lorsque $d=-1$, ainsi que J.-L. Colliot-Th\'el\`ene pour
plusieurs suggestions et discussions int\'eressantes sur cet article.

\end{document}